\documentclass[a4paper,11pt,reqno]{amsart}

\usepackage[utf8]{inputenc}
\usepackage[T1]{fontenc}
\usepackage{hyperref}
\usepackage{xcolor}
\hypersetup{
    colorlinks,
    linkcolor={red!50!black},
    citecolor={blue!50!black},
    urlcolor={blue!80!black}
}

\usepackage[T1]{fontenc}
\usepackage{lmodern}

\usepackage{microtype}
\usepackage{MnSymbol}
\usepackage{textcomp}
\usepackage[only,llbracket,rrbracket]{stmaryrd}

\usepackage[mathcal]{euscript}
\let\mathcaltmp\mathcal
\usepackage{mathrsfs}
\let\mathcal\mathscr
\let\mathscr\mathcaltmp

\makeatletter
\newcommand{\eqnum}{\refstepcounter{equation}\textup{\tagform@{\theequation}}}
\makeatother

\makeatletter \@addtoreset{equation}{section} \makeatother
\renewcommand{\theequation}{\thesection.\arabic{equation}}

\newtheorem{thm}[equation]{Theorem}
\newtheorem{thmX}{Theorem}
\newtheorem{corX}[thmX]{Corollary}

\newtheorem*{thm*}{Theorem}
\newtheorem{lem}[equation]{Lemma}
\newtheorem{cor}[equation]{Corollary}
\newtheorem{prop}[equation]{Proposition}
\newtheorem*{defthm*}{Definition/Theorem}

\theoremstyle{definition}
\newtheorem{defn}[equation]{Definition}
\newtheorem{rem}[equation]{Remark}
\newtheorem{exam}[equation]{Example}
\newtheorem{constr}[equation]{Construction}
\newtheorem{varnt}[equation]{Variant}
\newtheorem{notat}[equation]{Notation}

\newtheorem*{exam*}{Example}

\newcommand\arXiv[1]{\href{http://arxiv.org/abs/#1}{arXiv:#1}}

\usepackage{xparse}

\usepackage{xspace}

\newcommand{\changelocaltocdepth}[1]{%
  \addtocontents{toc}{\protect\setcounter{tocdepth}{#1}}%
  \setcounter{tocdepth}{#1}}

\newcommand{\nc}{\newcommand}
\nc{\renc}{\renewcommand}
\nc{\ssec}{\subsection}
\nc{\sssec}{\subsubsection}
\nc{\on}{\operatorname}
\nc{\term}[1]{#1\xspace}

\usepackage{tikz}
\usetikzlibrary{matrix}
\usepackage{tikz-cd}

\tikzset{
  commutative diagrams/.cd,
  arrow style=tikz,
  diagrams={>=latex}}
\makeatletter
\tikzset{
  column sep/.code=\def\pgfmatrixcolumnsep{\pgf@matrix@xscale*(#1)},
  row sep/.code   =\def\pgfmatrixrowsep{\pgf@matrix@yscale*(#1)},
  matrix xscale/.code=%
    \pgfmathsetmacro\pgf@matrix@xscale{\pgf@matrix@xscale*(#1)},
  matrix yscale/.code=%
    \pgfmathsetmacro\pgf@matrix@yscale{\pgf@matrix@yscale*(#1)},
  matrix scale/.style={/tikz/matrix xscale={#1},/tikz/matrix yscale={#1}}}
\def\pgf@matrix@xscale{1}
\def\pgf@matrix@yscale{1}
\makeatother

\usepackage[inline]{enumitem}
\setlist[enumerate,1]{label={(\alph*)},itemsep=\parskip}

\newlist{thmlist}{enumerate}{1}
\setlist[thmlist,1]{
  label={\em(\roman*)}, ref={(\roman*)},
  itemsep=0.5em,
  align=right,widest=vi)}

\newlist{thmlistbis}{enumerate}{1}
\setlist[thmlistbis,1]{
  label={\em(\roman*~\textit{bis})},
  ref={(\roman*}~\textit{bis}\upshape{)},
  itemsep=0.5em,
  leftmargin=0pt, align=right, widest=vi)}

\newlist{defnlist}{enumerate}{2}
\setlist[defnlist,1]{
  label={(\roman*)}, ref={(\roman*)},
  itemsep=0.5em,
  topsep=0em,
  leftmargin=*,
  align=left, widest=vi)}

\setlist[defnlist,2]{
  label={(\alph*)}, ref={(\alph*)},
  itemsep=0.75em,
  labelsep=0em,labelindent=0em,leftmargin=*,align=left,widest=vi),
  topsep=0.75em}

\newlist{inlinelist}{enumerate*}{1}
\setlist[inlinelist,1]{label={(\alph*)}}

\newlist{inlinedefnlist}{enumerate*}{1}
\definecolor{green}{HTML}{38550C}
\setlist[inlinedefnlist,1]{label={\color{green}(\roman*)}}

\nc{\cA}{\ensuremath{\mathcal{A}}\xspace}
\nc{\cB}{\ensuremath{\mathcal{B}}\xspace}
\nc{\cC}{\ensuremath{\mathcal{C}}\xspace}
\nc{\cD}{\ensuremath{\mathcal{D}}\xspace}
\nc{\cE}{\ensuremath{\mathcal{E}}\xspace}
\nc{\cF}{\ensuremath{\mathcal{F}}\xspace}
\nc{\cG}{\ensuremath{\mathcal{G}}\xspace}
\nc{\cH}{\ensuremath{\mathcal{H}}\xspace}
\nc{\cI}{\ensuremath{\mathcal{I}}\xspace}
\nc{\cJ}{\ensuremath{\mathcal{J}}\xspace}
\nc{\cK}{\ensuremath{\mathcal{K}}\xspace}
\nc{\cL}{\ensuremath{\mathcal{L}}\xspace}
\nc{\cM}{\ensuremath{\mathcal{M}}\xspace}
\nc{\cN}{\ensuremath{\mathcal{N}}\xspace}
\nc{\cO}{\ensuremath{\mathcal{O}}\xspace}
\nc{\cP}{\ensuremath{\mathcal{P}}\xspace}
\nc{\cQ}{\ensuremath{\mathcal{Q}}\xspace}
\nc{\cR}{\ensuremath{\mathcal{R}}\xspace}
\nc{\cS}{\ensuremath{\mathcal{S}}\xspace}
\nc{\cT}{\ensuremath{\mathcal{T}}\xspace}
\nc{\cU}{\ensuremath{\mathcal{U}}\xspace}
\nc{\cV}{\ensuremath{\mathcal{V}}\xspace}
\nc{\cW}{\ensuremath{\mathcal{W}}\xspace}
\nc{\cX}{\ensuremath{\mathcal{X}}\xspace}
\nc{\cY}{\ensuremath{\mathcal{Y}}\xspace}
\nc{\cZ}{\ensuremath{\mathcal{Z}}\xspace}

\nc{\sA}{\ensuremath{\mathscr{A}}\xspace}
\nc{\sB}{\ensuremath{\mathscr{B}}\xspace}
\nc{\sC}{\ensuremath{\mathscr{C}}\xspace}
\nc{\sD}{\ensuremath{\mathscr{D}}\xspace}
\nc{\sE}{\ensuremath{\mathscr{E}}\xspace}
\nc{\sF}{\ensuremath{\mathscr{F}}\xspace}
\nc{\sG}{\ensuremath{\mathscr{G}}\xspace}
\nc{\sH}{\ensuremath{\mathscr{H}}\xspace}
\nc{\sI}{\ensuremath{\mathscr{I}}\xspace}
\nc{\sJ}{\ensuremath{\mathscr{J}}\xspace}
\nc{\sK}{\ensuremath{\mathscr{K}}\xspace}
\nc{\sL}{\ensuremath{\mathscr{L}}\xspace}
\nc{\sM}{\ensuremath{\mathscr{M}}\xspace}
\nc{\sN}{\ensuremath{\mathscr{N}}\xspace}
\nc{\sO}{\ensuremath{\mathscr{O}}\xspace}
\nc{\sP}{\ensuremath{\mathscr{P}}\xspace}
\nc{\sQ}{\ensuremath{\mathscr{Q}}\xspace}
\nc{\sR}{\ensuremath{\mathscr{R}}\xspace}
\nc{\sS}{\ensuremath{\mathscr{S}}\xspace}
\nc{\sT}{\ensuremath{\mathscr{T}}\xspace}
\nc{\sU}{\ensuremath{\mathscr{U}}\xspace}
\nc{\sV}{\ensuremath{\mathscr{V}}\xspace}
\nc{\sW}{\ensuremath{\mathscr{W}}\xspace}
\nc{\sX}{\ensuremath{\mathscr{X}}\xspace}
\nc{\sY}{\ensuremath{\mathscr{Y}}\xspace}
\nc{\sZ}{\ensuremath{\mathscr{Z}}\xspace}

\nc{\bA}{\ensuremath{\mathbf{A}}\xspace}
\nc{\bB}{\ensuremath{\mathbf{B}}\xspace}
\nc{\bC}{\ensuremath{\mathbf{C}}\xspace}
\nc{\bD}{\ensuremath{\mathbf{D}}\xspace}
\nc{\bE}{\ensuremath{\mathbf{E}}\xspace}
\nc{\bF}{\ensuremath{\mathbf{F}}\xspace}
\nc{\bG}{\ensuremath{\mathbf{G}}\xspace}
\nc{\bH}{\ensuremath{\mathbf{H}}\xspace}
\nc{\bI}{\ensuremath{\mathbf{I}}\xspace}
\nc{\bJ}{\ensuremath{\mathbf{J}}\xspace}
\nc{\bK}{\ensuremath{\mathbf{K}}\xspace}
\nc{\bL}{\ensuremath{\mathbf{L}}\xspace}
\nc{\bM}{\ensuremath{\mathbf{M}}\xspace}
\nc{\bN}{\ensuremath{\mathbf{N}}\xspace}
\nc{\bO}{\ensuremath{\mathbf{O}}\xspace}
\nc{\bP}{\ensuremath{\mathbf{P}}\xspace}
\nc{\bQ}{\ensuremath{\mathbf{Q}}\xspace}
\nc{\bR}{\ensuremath{\mathbf{R}}\xspace}
\nc{\bS}{\ensuremath{\mathbf{S}}\xspace}
\nc{\bT}{\ensuremath{\mathbf{T}}\xspace}
\nc{\bU}{\ensuremath{\mathbf{U}}\xspace}
\nc{\bV}{\ensuremath{\mathbf{V}}\xspace}
\nc{\bW}{\ensuremath{\mathbf{W}}\xspace}
\nc{\bX}{\ensuremath{\mathbf{X}}\xspace}
\nc{\bY}{\ensuremath{\mathbf{Y}}\xspace}
\nc{\bZ}{\ensuremath{\mathbf{Z}}\xspace}

\nc{\bbA}{\ensuremath{\mathbb{A}}\xspace}
\nc{\bbB}{\ensuremath{\mathbb{B}}\xspace}
\nc{\bbC}{\ensuremath{\mathbb{C}}\xspace}
\nc{\bbD}{\ensuremath{\mathbb{D}}\xspace}
\nc{\bbE}{\ensuremath{\mathbb{E}}\xspace}
\nc{\bbF}{\ensuremath{\mathbb{F}}\xspace}
\nc{\bbG}{\ensuremath{\mathbb{G}}\xspace}
\nc{\bbH}{\ensuremath{\mathbb{H}}\xspace}
\nc{\bbI}{\ensuremath{\mathbb{I}}\xspace}
\nc{\bbJ}{\ensuremath{\mathbb{J}}\xspace}
\nc{\bbK}{\ensuremath{\mathbb{K}}\xspace}
\nc{\bbL}{\ensuremath{\mathbb{L}}\xspace}
\nc{\bbM}{\ensuremath{\mathbb{M}}\xspace}
\nc{\bbN}{\ensuremath{\mathbb{N}}\xspace}
\nc{\bbO}{\ensuremath{\mathbb{O}}\xspace}
\nc{\bbP}{\ensuremath{\mathbb{P}}\xspace}
\nc{\bbQ}{\ensuremath{\mathbb{Q}}\xspace}
\nc{\bbR}{\ensuremath{\mathbb{R}}\xspace}
\nc{\bbS}{\ensuremath{\mathbb{S}}\xspace}
\nc{\bbT}{\ensuremath{\mathbb{T}}\xspace}
\nc{\bbU}{\ensuremath{\mathbb{U}}\xspace}
\nc{\bbV}{\ensuremath{\mathbb{V}}\xspace}
\nc{\bbW}{\ensuremath{\mathbb{W}}\xspace}
\nc{\bbX}{\ensuremath{\mathbb{X}}\xspace}
\nc{\bbY}{\ensuremath{\mathbb{Y}}\xspace}
\nc{\bbZ}{\ensuremath{\mathbb{Z}}\xspace}

\nc{\mrm}[1]{\ensuremath{\mathrm{#1}}\xspace}
\nc{\mfr}[1]{\ensuremath{\mathfrak{#1}}\xspace}
\nc{\mit}[1]{\ensuremath{\mathit{#1}}\xspace}
\nc{\mbf}[1]{\ensuremath{\mathbf{#1}}\xspace}
\nc{\mcal}[1]{\ensuremath{\mathcal{#1}}\xspace}
\nc{\msc}[1]{\ensuremath{\mathscr{#1}}\xspace}

\nc{\sub}{\subseteq}
\nc{\too}{\longrightarrow}
\nc{\hook}{\hookrightarrow}
\nc{\hooklongrightarrow}{\lhook\joinrel\longrightarrow}
\nc{\hooklong}{\hooklongrightarrow}
\nc{\hooklongleftarrow}{\longleftarrow\joinrel\rhook}
\nc{\twoheadlongrightarrow}{\relbar\joinrel\twoheadrightarrow}
\nc{\longrightleftarrows}{\ \raisebox{0.3ex}{\(\mathrel{\substack{\xrightarrow{\rule{1em}{0em}} \\[-1ex] \xleftarrow{\rule{1em}{0em}}}}\)}\ }

\renc{\ge}{\geqslant}
\renc{\le}{\leqslant}

\nc{\id}{\mathrm{id}}

\DeclareMathOperator{\rk}{\mathrm{rk}}
\DeclareMathOperator{\Hom}{\on{Hom}}
\nc{\uHom}{\underline{\smash{\Hom}}}

\DeclareMathOperator{\Aut}{\on{Aut}}
\DeclareMathOperator{\End}{\on{End}}
\DeclareMathOperator{\Sym}{\on{Sym}}
\nc{\uEnd}{\underline{\smash{\End}}}

\nc{\colim}{\varinjlim}
\renc{\lim}{\varprojlim}
\nc{\Cofib}{\on{Cofib}}
\nc{\Fib}{\on{Fib}}
\nc{\initial}{\varnothing}
\nc{\op}{\mathrm{op}}

\DeclareMathOperator*{\fibprod}{\times}

\renc{\setminus}{\smallsetminus}

\usepackage{mathtools}
\DeclarePairedDelimiter\abs{\lvert}{\rvert}%

\newcommand{\thmref}[1]{Theorem~\ref{#1}}

\newcommand{\secref}[1]{Sect.~\ref{#1}}
\newcommand{\ssecref}[1]{Subsect. ~\ref{#1}}
\newcommand{\sssecref}[1]{(\ref{#1})}

\newcommand{\propref}[1]{Proposition~\ref{#1}}
\newcommand{\corref}[1]{Corollary~\ref{#1}}

\newcommand{\remref}[1]{Remark~\ref{#1}}
\newcommand{\defnref}[1]{Definition~\ref{#1}}

\renewcommand{\eqref}[1]{(\ref{#1})}
\newcommand{\constrref}[1]{Construction~\ref{#1}}

\newcommand{\examref}[1]{Example~\ref{#1}}

\newcommand{\notatref}[1]{Notation~\ref{#1}}
\newcommand{\itemref}[1]{\ref{#1}}

\nc{\A}{\bA}
\renc{\P}{\bP}
\nc{\Spec}{\on{Spec}}
\nc{\D}{\on{\mbf{D}}}
\nc{\SH}{\on{\mbf{SH}}}
\nc{\DM}{\on{\mbf{DM}}}
\nc{\MGLmod}{\on{\mbf{D}_\MGL}}
\nc{\MGLmodQ}{\on{\mbf{D}_{\MGL,\Q}}}
\nc{\cat}{\mrm{1}}
\nc{\AHR}{\mrm{AHR}}
\nc{\geom}{\mrm{T}}
\nc{\Dqc}{\on{\mbf{D}}_{\mrm{qc}}}
\nc{\bDelta}{\mathbf{\Delta}}
\nc{\Cech}{\textnormal{\v{C}}}
\nc{\Dperf}{\on{\mbf{D}}_{\mrm{perf}}}
\nc{\Coh}{\on{Coh}}
\nc{\Qcoh}{\on{Qcoh}}
\nc{\Dcoh}{\on{\mbf{D}}_{\mrm{coh}}}
\nc{\uCoh}{\underline{\smash{\Coh}}}
\nc{\Higgs}{\underline{\smash{\on{Higgs}}}}
\nc{\cl}{{\mrm{cl}}}
\nc{\Bl}{\on{Bl}}
\nc{\vir}{\mrm{vir}}
\nc{\CH}{\on{A}}
\nc{\et}{\mrm{\acute{e}t}}
\renc{\H}{\on{H}}
\nc{\BM}{\mrm{BM}}
\nc{\Z}{\bZ}
\nc{\Q}{\bQ}
\nc{\K}{{\on{K}}}
\nc{\KGL}{\mrm{KGL}}
\nc{\MGL}{\mrm{MGL}}
\nc{\KB}{\K^{\mrm{B}}}
\nc{\G}{{\on{G}}}
\nc{\KH}{\mrm{KH}}
\nc{\Ket}{\K^{\et}}
\nc{\KHet}{\KH^{\et}}
\nc{\Get}{\G^{\et}}
\nc{\rightrightrightarrows}
  {\mathrel{\substack{\textstyle\rightarrow\\[-0.6ex]
   \textstyle\rightarrow \\[-0.6ex]
   \textstyle\rightarrow}}}
\nc{\rightrightrightrightarrows}
  {\mathrel{\substack{\textstyle\rightarrow\\[-0.6ex]
   \textstyle\rightarrow \\[-0.6ex]
   \textstyle\rightarrow \\[-0.6ex]
   \textstyle\rightarrow}}}
\nc{\Einfty}{{\sE_\infty}}
\renc{\sp}{\mrm{sp}}
\nc{\Td}{\on{Td}}
\nc{\ch}{\on{ch}}
\nc{\RGamma}{R\Gamma}
\nc{\red}{\mrm{red}}
\nc{\der}{{\mrm{der}}}
\nc{\Mod}{{\mrm{Mod}}}
\nc{\Gr}{{\on{Gr}}}
\nc{\Ind}{\on{Ind}}
\nc{\form}{\widehat}
\nc{\R}{\bR}
\renc{\L}{\bL}
\nc{\otimesL}{\mathchoice{\overset{\bL}{\otimes}}{\otimes^\bL}{\otimes^\bL}{\otimes^\bL}}
\nc{\fibprodR}{\fibprod^R}
\nc{\uRHom}{\bR\uHom}
\nc{\GL}{\mrm{GL}}
\nc{\SL}{\mrm{SL}}
\nc{\SW}{\on{SW}}
\nc{\Vect}{\on{Vect}}
\nc{\Fun}{\on{Fun}}
\nc{\vb}[1]{\langle #1\rangle}
\nc{\loc}{\mrm{loc}}
\nc{\fix}{\mrm{fix}}
\nc{\mov}{\mrm{mov}}
\nc{\cms}{\mrm{cms}}
\nc{\V}{\bV}
\nc{\Gm}{{\bG_m}}
\nc{\pt}{\mrm{pt}}
\nc{\mot}{\mrm{mot}}
\nc{\an}{\mrm{an}}
\nc{\St}{\mrm{St}}
\nc{\pr}{\mrm{pr}}
\nc{\C}{\on{C}}
\nc{\Chom}{\mrm{C}_\bullet}
\nc{\Chomhat}{\widehat{\mrm{C}}_\bullet}
\nc{\Ccoh}{\mrm{C}^\bullet}
\nc{\Ccohhat}{\widehat{\mrm{C}}^\bullet}
\nc{\CBM}{\mrm{C}^{\BM}_\bullet}
\nc{\uAut}{\underline{\Aut}}
\nc{\per}{\vb{\ast}} % "periodization"
\nc{\Lis}{\mrm{Lis}}
\nc{\aff}{{\mrm{aff}}}
\nc{\dR}{{\mrm{dR}}}
\nc{\Pic}{{\on{Pic}}}
\nc{\uGrp}{\underline{\smash{\mrm{Grp}}}}
\nc{\uPerf}{\underline{\smash{\mrm{Perf}}}}
\nc{\uPic}{\underline{\smash{\Pic}}}
\nc{\uMap}{\underline{\smash{\mrm{Map}}}}
\nc{\uDiv}{\underline{\smash{\mrm{Div}}}}
\nc{\uPair}{\underline{\smash{\mrm{Pair}}}}
\nc{\dash}{\textnormal{-}}
\nc{\nilp}{\mrm{nilp}}
\nc{\Rep}{\on{R}}
\nc{\dashmod}{\dash\mbf{mod}}
\nc{\bLambda}{\mbf{\Lambda}}
\renc{\max}{\mrm{max}}
\nc{\un}{\mbf{1}}
\nc{\Stk}{\mrm{Stk}}
\nc{\dStk}{\mrm{dStk}}

\nc{\scr}{\term{derived commutative ring}}
\nc{\scrs}{\term{derived commutative rings}}

\nc{\inftyCat}{\term{$\infty$-category}}
\nc{\inftyCats}{\term{$\infty$-categories}}

\nc{\inftyGrpd}{\term{$\infty$-groupoid}}
\nc{\inftyGrpds}{\term{$\infty$-groupoids}}

\nc{\dA}{\term{derived Artin}}

\title{The stacky concentration theorem\vspace{-2mm}}
\author[D. Aranha]{Dhyan Aranha}
\author[A.\,A. Khan]{Adeel A. Khan}
\author[A. Latyntsev]{Alexei Latyntsev}
\author[H. Park]{Hyeonjun Park}
\author[C. Ravi]{Charanya Ravi}
\date{2025-01-17}

\makeatletter
\def\l@subsection{\@tocline{2}{0pt}{4pc}{6pc}{}}
\def\l@subsubsection{\@tocline{3}{0pt}{8pc}{8pc}{}}
\makeatother

\begin{document}

\begin{abstract}
  We give a sufficient criterion for the Chow or algebraic bordism groups of an algebraic stack, localized at a set of Chern classes of line bundles, to be concentrated in some closed substack.
  This is a vast generalization of the torus fixed-point localization theorem in equivariant intersection theory, which is the special case of the stack quotient of a scheme $X$ by an action of a torus $T$.
  Taking on the one hand an algebraic stack in place of $X$, we deduce a generalization of torus localization to algebraic stacks.
  Taking on the other hand any algebraic group $G$ instead of $T$, we obtain a localization theorem in $G$-equivariant intersection theory.
  \vspace{-5mm}
\end{abstract}

\maketitle

\renewcommand\contentsname{\vspace{-1cm}}
\tableofcontents

\setlength{\parindent}{0em}
\parskip 0.75em

\thispagestyle{empty}

%%%%%%%%%%%%%%%%%%%%%%%%%%%%%%%%%%%%%%%%%%%%%%%%%%%%%%%%%%%%%%%%%%%%%%%%%%%

\changelocaltocdepth{2}

\section*{Introduction}

  \subsection{Main results}

    The starting point of this paper is the \emph{concentration} theorem\footnote{%
      This is often simply called ``the localization theorem''.
      The less ambiguous terminology is due to R.\,W.~Thomason \cite{ThomasonLefschetz} as far as we know.
    } in equivariant intersection theory.
    Given a split torus $T$ over a field $k$ and a $k$-scheme $X$ with $T$-action, we denote by $\CH^T_*(X)$ the $T$-equivariant Chow group of $X$ with rational coefficients.
    Then we have (see \cite[Thm.~1]{EdidinGrahamLocalization}):
    
    \begin{thm}\label{thm:intro/classical conc}
      Let $X$ be a scheme of finite type over a field with an action of a split torus $T$.
      Let $i : X^T \hook X$ denote the inclusion of the $T$-fixed locus.
      Then the push-forward map on $T$-equivariant Chow groups induces an isomorphism
      \begin{equation*}
        i_* : \CH^T_*(X^T)_\loc \to \CH^T_*(X)_\loc
      \end{equation*}
      where the localization is at the set of first Chern classes of nontrivial $1$-dimensional representations of $T$.
    \end{thm}

    This result forms the core of equivariant localization theory; for instance, it leads immediately to the localization formula of Atiyah and Bott \cite{AtiyahBott} when $X$ is smooth and to the Bott residue formula \cite{Bott} when $X$ is smooth and proper.
    
    In this paper we prove a new type of concentration theorem for general algebraic stacks.
    Let $\CH_*(\sX)$ denote the cycle groups (with rational coefficients) defined in \cite{KhanVirtual} for any algebraic stack $\sX$ of finite type over $k$.\footnote{%
      That is, $\CH_n(\sX)$ is defined as the motivic Borel--Moore homology group $\H^\BM_{2n}(\sX; \Q^\mot(-n))$ of \cite{KhanVirtual}, and $\CH_*(\sX)$ is the direct sum over $n\in\Z$.
    }
    For the quotient stack $\sX=[X/T]$, this agrees with the $T$-equivariant cycle groups $\CH_*^T(X)$.\footnote{%
      By \cite[Cor.~6.5]{Equilisse}, we have $\CH_n(\sX) \simeq \CH_{n-\dim(T)}^T(X)$.
    }
    We claim (see \corref{cor:conc}):
    
    \begin{thmX}[Stacky concentration]\label{thm:intro/stacky conc}
      Let $\sX$ be an Artin stack of finite type over a field $k$ with affine stabilizers and $\sZ\sub \sX$ a closed substack.
      Let $\Sigma$ be a set of line bundles on $\sX$ such that for every geometric point $x$ of $\sX\setminus\sZ$, there exists $\cL(x) \in \Sigma$ whose restriction to $B\uAut_\sX(x)$ is trivial.
      Then push-forward along $i : \sZ \hook \sX$ induces an isomorphism
      \[
        i_* : \CH_*(\sZ)[\Sigma^{-1}] \to \CH_*(\sX)[\Sigma^{-1}]
      \]
      where $\cL \in \Sigma$ acts via the first Chern class.
    \end{thmX}

    From this we derive the following generalization of \thmref{thm:intro/classical conc} from schemes to Artin stacks (see \thmref{thm:excitor}).
    Given an Artin stack $X$ over $k$ with $T$-action, we write $\CH^T_*(X)$ for the $T$-equivariant cycle groups of $X$ (with rational coefficients).\footnote{%
      That is, $\CH^T_n(X)$ is defined as the relative motivic Borel--Moore homology group $\H^\BM_{2n}([X/T]_{/BT}; \Q^\mot(-n))$ of \cite{KhanVirtual}, and $\CH^T_*(X)$ is the direct sum over $n\in\Z$.
      For $X$ a scheme or algebraic space, this agrees with the Edidin--Graham equivariant cycle groups again by \cite[Cor.~6.5]{Equilisse}.
    }

    \begin{thmX}[Torus concentration]\label{thm:intro/torus conc}
      Let $X$ be an Artin stack of finite type over $k$ with $T$-action.
      Let $Z\sub X$ be a $T$-invariant closed substack away from which every point has $T$-stabilizer properly contained in $T$.\footnote{%
        Informally speaking, the condition on $Z$ means that it contains every $T$-fixed point.
        See \ssecref{ssec:gstab} for the notion of $T$-stabilizers (which are not the same as the stabilizers of $[X/T]$ when $X$ is not an algebraic space).
      }
      Then push-forward along $i : Z \hook X$ induces an isomorphism
      \begin{equation}\label{eq:liturgiological}
        \CH^T_*(Z)_\loc \to \CH^T_*(X)_\loc
      \end{equation}
      where the localization is as in \thmref{thm:intro/classical conc}.
    \end{thmX}

  \subsection{Examples, variants, and generalizations}

    \subsubsection{Deligne--Mumford stacks}

      When $X$ is a Deligne--Mumford stack, \thmref{thm:intro/torus conc} was proven by A.~Kresch in the case where the torus $T$ is of rank $1$ and the field $k$ is algebraically closed (see \cite[Thm.~5.3.5]{Kresch}).
      
      Note that, at least when $X$ is separated, $\CH^T_*(X) \simeq \CH_*([X/T])$ agrees with Kresch's cycle groups \cite{Kresch}.
      Indeed, $X$ admits a coarse moduli space $M$ inheriting the $T$-action and both cycle groups are canonically isomorphic to $\CH^T_*(M)$ (recall that we are working with rational coefficients).
      
    \subsubsection{Artin stacks with finite stabilizers}

      Suppose $X$ is an Artin stack with finite stabilizers; recall that when $k$ is of positive characteristic, $X$ need not be Deligne--Mumford.
      For example, this applies to moduli stacks of stable objects in enumerative geometry, e.g. moduli of stable maps in Gromov--Witten theory and moduli of stable sheaves in Donaldson--Thomas theory.

      As long as $X$ is separated (or more generally has finite inertia), $\CH^T_*(X)$ still agrees with Kresch's cycle groups for the same reason as in the Deligne--Mumford case.

    \subsubsection{Global quotient stacks}

      Suppose $G$ is a linear algebraic group over $k$ and $X$ is an algebraic space of finite type over $k$ with commuting actions of $G$ and $T$.
      Suppose $Z \sub X$ is a closed subspace invariant under both actions which contains all points whose $T$-stabilizer is the whole torus $T$.
      Then \thmref{thm:intro/torus conc} applied to $\sX=[X/G]$ yields the isomorphism
      \begin{equation}
        i_* : \CH^{G\times T}_*(Z)_\loc \to \CH^{G\times T}_*(X)_\loc
      \end{equation}
      on $G\times T$-equivariant Chow groups.
      This generalizes a concentration theorem proven by A.~Minets \cite{Minets} under a much more restrictive technical condition.

    \subsubsection{Locally of finite type stacks}

      We will also prove a generalization of \thmref{thm:intro/torus conc} to \emph{locally} of finite type Artin stacks.
      In fact, the result is false in this generality with the na\"ive definition of $\CH^T_*(X)_\loc$.
      In \ssecref{ssec:nonqc} we will give another definition of the latter that agrees with the na\"ive localization in the quasi-compact case, but for which concentration does hold.
      For example, this applies to the moduli stack of Higgs sheaves on a curve, which admits a canonical scaling action of $T=\bG_m$.
      This was one of the motivating examples\footnote{%
        We thank Y.~Soibelman for suggesting it; see \cite{FedorovSoibelmanSoibelman} for an analogous analysis at the level of virtual motives.
      } (see \thmref{thm:Higgs}):

      \begin{corX}\label{cor:intro/Higgs}
        Let $k$ be a field and $C$ a smooth proper and geometrically connected curve over $k$.
        Denote by $\Higgs$ the moduli stack of Higgs sheaves on $C$ and by $\bLambda$ the closed substack parametrizing nilpotent Higgs sheaves.
        Then the map
        \begin{equation*}
          \CH^{T}_*(\bLambda)_\loc \to \CH^{T}_*(\Higgs)_\loc
        \end{equation*}
        is invertible.
      \end{corX}

      This generalizes a result of A.~Minets \cite[Cor.~4.3]{Minets}, who considered the substack of \emph{torsion} (rank zero) Higgs sheaves.

    \subsubsection{Localizing K-theory classes}

      The following variant of \thmref{thm:intro/stacky conc} can be more useful when there is not a sufficient supply of $1$-dimensional representations.

      With notation as in \thmref{thm:intro/stacky conc}, let $\Sigma \sub \K_0(\sX)$ be a set of K-theory classes\footnote{%
        Here, $\K_0(\sX)$ refers to the Grothendieck group of perfect complexes on $\sX$.
      } such that for every geometric point $x$ of 
      $\sX\setminus\sZ$, there exists an $\alpha(x) \in \Sigma$ whose restriction to $B\uAut_\sX(x)$ is trivial.
      Let $\hat{\CH}_*(\sX)$ denote the \emph{product} of the cycle groups $\CH_n(\sX)$ over $n\in\Z$.
      Then we will prove that $i_*$ induces an isomorphism
      \begin{equation}
        i_* : \hat{\CH}_*(\sZ)[\Sigma^{-1}] \to \hat{\CH}_*(\sX)[\Sigma^{-1}]
      \end{equation}
      where $\cL \in \Sigma$ acts via the Chern character.
      See \thmref{thm:conc master rational}.

    \subsubsection{General algebraic group actions}

      We can specialize stacky concentration to quotients by arbitrary algebraic group actions $[X/G]$.
      We get the following generalization of \thmref{thm:intro/torus conc} (see \corref{cor:equiv conc}):

      \begin{corX}[Equivariant concentration]\label{cor:intro/equiv conc}
        Let $G$ be an algebraic group acting on an Artin stack $X$ of finite type over a field $k$ with affine stabilizers.
        Let $\Sigma \sub \K_0(BG)$ be a subset of nonzero elements.
        Let $Z\sub X$ be a $G$-invariant closed substack containing every point $x$ of $X$ such that no element of $\Sigma$ is sent to zero by $\K_0(BG) \to \K_0(B\St^G_X(x))$.
        Then
        \[
          i_* : \hat{\CH}^{G}_*(Z)[\Sigma^{-1}] \to \hat{\CH}^{G}_*(X)[\Sigma^{-1}]
        \]
        is an isomorphism.
      \end{corX}

    \subsubsection{Oriented Borel--Moore homology theories}

      Our methods work for general oriented Borel--Moore type homology theories.
      This includes usual Borel--Moore homology (over $k=\bC$) and $\ell$-adic Borel--Moore homology (over a field in which $\ell$ is invertible).

      We also prove concentration for algebraic bordism, namely for the stacky and equivariant (higher) algebraic bordism theories defined in \cite[Ex.~2.12]{KhanVirtual} and \cite[\S 8]{Equilisse}.
      The analogous statement for Levine--Morel algebraic bordism was proven for smooth and projective varieties over a field of characteristic zero by A.~Krishna \cite{KrishnaTorusLocalization}, and generalized recently in \cite{KiemPark} to (possibly singular) quasi-projective schemes.

      Our methods also apply to G-theory (= algebraic K-theory of coherent sheaves).
      In this paper, we only consider étale G-theory and Borel-equivariant G-theory (in the sense of \cite[\S 7]{Equilisse}).
      However, the case of ``genuine'' G-theory can be treated by the same method and will be done in detail in \cite{Kloc}.

    \subsubsection{Categorification}

      In a sequel \cite{Concat} to this paper, two of the authors proved a categorification of the stacky concentration theorem.
      For example, consider the stable \inftyCat $\DM(X)_\Q$ of ($\Q$-linear) motivic sheaves on an Artin stack $X$.
      In the situation of \thmref{thm:intro/stacky conc}, we have the $*$-direct image functor $i_* : \DM(\sZ)_\Q \to \DM(\sX)_\Q$.
      Categorified concentration asserts that $i_*$ induces an equivalence of stable \inftyCats
      \begin{equation}
        i_* : \DM(\sZ)_\Q[\Sigma^{-1}] \simeq \DM(\sX)_\Q[\Sigma^{-1}]
      \end{equation}
      where the localization is in an appropriate $\infty$-categorical sense.

      This result is used to obtain concentration for oriented \emph{cohomology} theories, such as motivic cohomology or Betti and $\ell$-adic cohomology, where it is the pullback map $i^* : \H^*(\sX) \to \H^*(\sZ)$ that induces an isomorphism after localization.

  \subsection{Proof}

    Let us briefly comment on the proof of the stacky concentration theorem (\thmref{thm:intro/stacky conc}).
    The cohomological formalism for intersection theory on stacks of \cite{KhanVirtual} plays an essential role in reducing this vastly general statement to a drastically simpler one.

    The cycle groups $\CH_n(\sX)$ are realized as the motivic cohomology groups
    \[
      \CH_n(\sX) = \H^{-2n}(\sX, f^!\Q^\mot(-n))
    \]
    where $\Q^\mot(-n) \in \DM(\Spec(k))$ is the Tate twist in Voevodsky motives, $f : \sX \to \Spec(k)$ is the projection, and $f^! : \DM(\Spec(k))_\Q \to \DM(\sX)_\Q$ is the compactly supported inverse image functor on stable \inftyCats of ($\Q$-linear) motivic sheaves.
    One can think of the right-hand side above as a motivic version of Borel--Moore homology: the cohomology of the dualizing sheaf, in the context of motivic sheaves.

    The $\CH_n(\sX)$ are just zeroth cohomology groups of the \emph{chain complexes} of derived global sections,
    \begin{equation*}
      \CBM(\sX)\vb{-n} := R\Gamma(\sX, f^! \bQ^\mot\vb{-n}),
    \end{equation*}
    which we regard as objects of the derived \inftyCat of chain complexes of $\Q$-vector spaces, where $\vb{-n} := (-n)[-2n]$.
    When $\sX$ is a scheme $X$, $\CBM(X)\vb{-n}$ coincides with the Bloch cycle complex $\mbf{z}_n(X)$ (see \cite[Ex.~2.10]{KhanVirtual}).
    One may thus think of $\CBM(\sX)\vb{-n}$ as a Bloch cycle complex for $\sX$.\footnote{Note that we do not know of any way to define the Bloch cycle complex $\mbf{z}_n(\sX)$ directly (i.e., without passing through motivic cohomology), despite two of the authors' attempts to do so.}
    
    For the proof of stacky concentration, our key tool is the localization triangle: for any closed substack $\sZ$ of $\sX$ with open complement $\sU = \sX\setminus\sZ$, there is an exact triangle of complexes
    \begin{equation*}
      \CBM(\sZ)\vb{-n} \to \CBM(\sX)\vb{-n} \to \CBM(\sU)\vb{-n}
    \end{equation*}
    This is a powerful extension of the right-exact localization sequence $\CH_n(\sZ) \to \CH_n(\sX) \to \CH_n(\sU) \to 0$.
    Using the formalism of higher algebra \cite{LurieHA}, the localization $[\Sigma^{-1}]$ may be performed at the level of complexes, and preserves such exact triangles.

    This device allows us to perform a series of reductions.
    First, to prove that push-forward along a closed immersion induces an isomorphism after $\Sigma$-localization, it suffices by localization to demonstrate $\Sigma$-acyclicity of motivic Borel--Moore chains on the complement.
    To show $\Sigma$-acyclicity of $\CBM(\sX)\vb{\ast}$ for some $\sX$, we may apply localization inductively to a stratification of $\sX$ by global quotient stacks to reduce to the case of $\sX = [X/\GL_m]$ a global quotient.
    By a variant of Thomason's ``reduction to the subtorus'' argument, which we prove for $\CBM(-)\vb{\ast}$ (see \thmref{thm:Weyl}), we can replace $\GL_m$ by its maximal subtorus $T$.
    This $T$-action is generically trivial, up to ``Morita equivalence'', so yet another application of the localization triangle reduces us to the following straightforward observation: if $\sX=[X/G]$ is the quotient of a scheme by the trivial action of a diagonalizable group scheme, such that there exists a line bundle $\sL \in \Sigma$ and a morphism $BG \to \sX$ with $c_1(\sL|_{BG}) = 0$, then the localization $\CBM(\sX)[\Sigma^{-1}]$ vanishes.
    We refer to \ssecref{ssec:musical} for the details.

  \subsection{Localization formulas}

    An immediate consequence of torus concentration is the Atiyah--Bott localization formula, which is an explicit computation of the inverse of the isomorphism $i_*$.
    In the situation of \thmref{thm:intro/torus conc}, suppose $X$ and $Z$ are smooth Artin stacks.
    Then the inclusion $i : Z \hook X$ is lci and we have a Gysin pull-back $i^! : \CH^T_*(X) \to \CH^T_*(Z)$ which satisfies the self-intersection formula
    \[ i^! \circ i_* (-) = (-) \cap e(N) \]
    where $e(N)$ is the top Chern class of the normal bundle.
    As long as the latter is invertible\footnote{%
      This is true when $Z$ has trivial $T$-action and finite stabilizers, and $N$ has no $T$-fixed part; see \cite[Prop.~1.3, Cor.~A.31]{virloc}.
    } in $\CH^*_T(Z)[\Sigma^{-1}]$, this immediately yields the localization formula
    \begin{equation}
      (i_*)^{-1} = i^!(-) \cap e(N)^{-1}.
    \end{equation}

    Moreover, the same strategy can be adapted to the ``virtual'' version, where $X$ and $Z$ are \emph{quasi-smooth}.
    This gives a conceptual proof of the virtual localization formula of \cite{GraberPandharipande} over arbitrary base fields, and shows that it holds without the global resolution and embeddability hypotheses required in \emph{op. cit}.
    We refer to \cite{virloc} for details.
    
  \subsection{Contents of the paper}

    We begin in \secref{sec:convent} by establishing some notation and conventions on algebraic stacks.
    In \ssecref{ssec:gstab} we introduce the $G$-stabilizer group at a point of an Artin stack $X$ with $G$-action; note that these are not the same as the stabilizers of the quotient stack $[X/G]$ when $X$ itself has nontrivial stabilizers.
    \ssecref{ssec:fixed} contains a review of various notions of $G$-fixed loci in the context of stacks, treated in detail in the companion paper \cite{virloc}.

    In \secref{sec:CohIT} we review the cohomological formalism for intersection theory on stacks developed in \cite{KhanVirtual}.
    We record the projective bundle formula in this context (\ssecref{ssec:PBF}).
    In \ssecref{ssec:subtorus} we establish an analogue of Thomason's ``reduction to the subtorus'', a very useful technique for reducing statements about $G$-equivariant intersection theory to the case where $G$ is a split torus.

    The stacky concentration theorem (\thmref{thm:intro/stacky conc}) is stated and proven in \secref{sec:conc}.
    We specialize to quotient stacks to derive equivariant concentration (\corref{cor:intro/equiv conc}) in \ssecref{ssec:gmast}.
    \secref{sec:torus} is devoted to the torus-equivariant case (\thmref{thm:intro/torus conc}), where we can derive a more explicit statement.
    We end with some examples, including concentration for the moduli stack of Higgs bundles on a curve with its canonical $\Gm$-scaling action (\corref{cor:intro/Higgs}).

  \subsection{Acknowledgments}

    We would like to thank Marc Levine and Yan Soibelman for their interest in this work, and Alexandre Minets for help with moduli stacks of Higgs sheaves.
    We would also like to thank Tasuki Kinjo who participated in this project in the initial stages, as well as the anonymous referees for very helpful comments.

    We acknowledge support from the ERC grant QUADAG (D.A.), the NSTC grant 110-2115-M-001-016-MY3 (A.A.K.), the DFG through SFB 1085 Higher Invariants (C.R.), and the EPSRC grant no EP/R014604/1 (A.A.K. and C.R.).
    We also thank the Isaac Newton Institute for Mathematical Sciences, Cambridge, for hospitality during the programme ``Algebraic K-theory, motivic cohomology and motivic homotopy theory'' where the final revisions of this paper were completed.
    C.R. is grateful to Max Planck Institute for Mathematics in Bonn for its hospitality and financial support at the time of writing the paper.
    This paper is part of a project that has received funding from the European Research Council (ERC) under the European Union's Horizon 2020 research and innovation programme (grant agreement No.~832833).

\changelocaltocdepth{2}

\section{Preliminaries on stacks}
\label{sec:convent}

  \ssec{Stacks}
  \label{ssec:larnax}

    We fix a base commutative ring $k$ and denote by $\Stk_k$ the \inftyCat of Artin stacks\footnote{%
      Note that, except in the introduction, all our stacks are implicitly \emph{higher} by default, i.e., they define sheaves of \inftyGrpds.
    } that are locally of finite type over $k$ and have quasi-compact and separated diagonal.
    Given an Artin stack $S \in \Stk_k$, we denote by $\Stk_S$ the \inftyCat of locally of finite type Artin stacks over $S$ with quasi-compact and separated diagonal.

    Let us briefly review the definitions; we refer to \cite[\S 4.2]{GaitsgoryStacks} or \cite[\S 3.1]{ToenSimplicial} for more details.
    A \emph{prestack} is a presheaf of \inftyGrpds on the site of $k$-schemes.
    A \emph{stack} is a prestack that satisfies hyperdescent with respect to the étale topology.

    A stack is \emph{$0$-Artin} if it is an algebraic space, i.e., if it has schematic and monomorphic (= $(-1)$-truncated) diagonal\footnote{%
      hence in particular takes values in sets (= $0$-truncated \inftyGrpds)
    } and admits a surjective étale morphism from a $k$-scheme.
    A stack is \emph{$n$-Artin}, for $n>0$, if it has $(n-1)$-representable diagonal and admits a surjective smooth morphism from a scheme.
    A stack is \emph{Artin} if it is $n$-Artin for some $n$.
    A stack is \emph{Deligne--Mumford} if it has representable (= $0$-representable) diagonal and admits a surjective étale morphism from a scheme, or equivalently if it is $1$-Artin with unramified diagonal.
    
    In the above definitions, a morphism of prestacks $f : X \to Y$ is called \emph{schematic}, resp. \emph{$(n-1)$-representable}, if for every scheme $S$ and every morphism $S \to Y$, the base change $X \fibprod_Y S$ is a scheme, resp. $(n-1)$-Artin.
    An $(n-1)$-representable morphism $f : X \to Y$ is \emph{étale} (resp. \emph{smooth}, \emph{flat}, \emph{surjective}), if for every scheme $S$ and every morphism $S \to Y$, the morphism of $(n-1)$-representable stacks $X \fibprod_Y S \to S$ is étale (resp. smooth, flat, surjective).

  \ssec{Points}

    A \emph{point} of a prestack $X$ is a field-valued point, i.e., a morphism $x : \Spec(k(x)) \to X$ where $k(x)$ is a field (which we call the \emph{residue field} at $x$).
    A \emph{geometric point} of $X$ is a field-valued point $x$ whose residue field $k(x)$ is an algebraic closure of a residue field of $k$ at a prime ideal.
    A morphism of stacks is \emph{surjective} if it is surjective on geometric points.

    The \emph{set of points} of $X$, denoted $\abs{X}$, is the colimit
    \[ \colim_{\kappa} \pi_0 X(\kappa), \]
    taken over fields $\kappa$, in the category of sets.
    Here $\pi_0 X(\kappa)$ is the set of connected components of the \inftyGrpd $X(\kappa)$, and given a field extension $\kappa' \to \kappa$, the corresponding transition arrow is the map induced by $X(\kappa') \to X(\kappa)$ on sets of connected components.
    When $X$ is $1$-Artin, $\abs{X}$ admits a canonical structure of topological space (see e.g. \cite[Tag~04XL]{Stacks}).

  \ssec{Stabilizers}
  \label{ssec:theosophic}

    Let $X$ be a $1$-Artin stack and $x$ a point.
    The \emph{stabilizer} at $x$ is the group algebraic space $\uAut_X(x)$ of automorphisms of $x$.
    This can be defined equivalently as the fibred product $\Spec(k(x)) \fibprod_X \Spec(k(x))$, the fibre of the diagonal $X \to X \times X$ over $(x,x)$, or the fibre of the projection of the inertia stack $I_X \to X$ over $x$.

    We say a $1$-Artin stack $X$ has \emph{affine stabilizers} if for every point $x$ of $X$, the stabilizer $\uAut_X(x)$ is affine.
    If $X$ has affine inertia or affine diagonal, then it has affine stabilizers.
    
    We say that $X$ has \emph{finite stabilizers} if for every point $x$ the stabilizer $\uAut_X(x)$ is finite (over $\Spec(k(x))$).
    When $X$ has quasi-compact diagonal, this is equivalent to the diagonal being moreover quasi-finite.
    If $X$ is Deligne--Mumford or has quasi-finite inertia, then it has finite stabilizers.
  
  \subsection{Stabilizers of group actions}
  \label{ssec:gstab}

    Suppose given an action of an algebraic group $G$ on an Artin stack $X$, in the sense of \cite{RomagnyGroupActions} (for $1$-Artin stacks) or \cite[\S 4.2]{KhanNCTS} (for higher Artin stacks).
    We will define the stabilizer of the action (``$G$-stabilizer'') at any point $x$ of $X$.
    When the stabilizer at $x$ of $X$ itself is trivial, this coincides with the stabilizer of the quotient stack $[X/G]$ at $x$.

    We fix an fppf group algebraic space $G$ over a $k$-scheme $S$.

    \begin{rem}\label{rem:Yana}
      Let $f : X \to Y$ be a morphism of Artin stacks over $S$.
      The relative inertia stack $I_{X/Y}$ is a group Artin stack over $X$ which fits into a cartesian square
      \[\begin{tikzcd}
        I_{X/Y} \ar{r}\ar{d}
        & I_{X/S} \ar{d}
        \\
        X \ar{r}
        & I_{Y/S} \fibprod_Y X
      \end{tikzcd}\]
      of group stacks over $X$.
      The lower horizontal arrow is the base change of the identity section $e : Y \to I_{Y/S}$.
      When $f$ is representable, $I_{X/Y} \to X$ is an isomorphism, i.e., $I_{X/S} \to I_{Y/S} \fibprod_Y X$ is a monomorphism of group stacks.
      See e.g. \cite[Tag~050P]{Stacks}.
    \end{rem}

    \begin{rem}
      Let $X$ be an Artin stack over $S$ with $G$-action and denote by $\sX = [X/G]$ the quotient stack.
      Applying \remref{rem:Yana} to the morphisms $X \twoheadrightarrow \sX$ and $\sX \to BG$, we get the cartesian squares of group stacks over $X$
      \[
        \begin{tikzcd}
          X \ar{r}\ar{d}
          & I_{X/S} \ar{r}\ar{d}
          & X \ar{d}
          \\
          X \ar{r}
          & I_{\sX/S} \fibprod_{\sX} X \ar{r}
          & G \fibprod_S X,
        \end{tikzcd}
      \]
      where $I_{X/\sX} \simeq X$ since $X \twoheadrightarrow \sX$ is representable and the right-hand vertical arrow is the identity section.
      For every $S$-scheme $A$ and every $A$-valued point $x$ of $X$, this gives rise to an exact sequence of group algebraic spaces over $A$
      \begin{equation}\label{eq:evvTWDTvdYQvm}
        1 \to \underline{\Aut}_X(x) \to \underline{\Aut}_{\sX}(x)
        \xrightarrow{\alpha_A} G_A
      \end{equation}
      where $G_A := G \fibprod_S A$ denotes the fibre of $G$ over $x$.
    \end{rem}

    \begin{defn}\label{defn:stabilizer}
      Let $X$ be an Artin stack over $S$ with $G$-action.
      For any scheme $A$ and every $A$-valued point $x$ of $X$, the \emph{$G$-stabilizer} (or \emph{stabilizer of the $G$-action}) at $x$ is an fppf sheaf of groups $\St^G_X(x)$ defined as the cokernel of the homomorphism $\uAut_X(x) \hook \uAut_\sX(x)$.
      Thus we have a short exact sequence
      \begin{equation}\label{eq:stabses}
        1 \to \underline{\Aut}_X(x) \to \underline{\Aut}_{\sX}(x) \to \St^G_X(x) \to 1
      \end{equation}
      of sheaves of groups over $A$.
      Note that $\St^G_X(x)$ can be regarded as a subgroup of $G_{A}$, since it is the image of $\alpha_A : \uAut_{\sX}(x) \to G_{A}$.
    \end{defn}

    \begin{rem}
      For a field-valued point $x : \Spec(k(x)) \to X$, the $G$-stabilizer $\St^G_X(x)$ is a group algebraic space.
      This follows from \cite[Exp.~V, Cor.~10.1.3]{SGA3}, since in this case $\uAut_X(x)$ is flat over $\Spec(k(x))$.
      Since $X$ has separated diagonal, $\St^G_X(x)$ is moreover a group \emph{scheme} by \cite[0B8F]{Stacks}.
    \end{rem}

    \begin{rem}\label{rem:noboh1b1}
      When $X$ has trivial stabilizers (i.e., is an algebraic space), then the $G$-stabilizer $\St^G_X(x)$ at a point $x$ is nothing else than $\uAut_{\sX}(x)$, the stabilizer at $x$ of the quotient stack $\sX = [X/G]$.
    \end{rem}

    \begin{rem}\label{rem:zymogenous}
      Let $X$ be an Artin stack over $S$ with $G$-action.
      Let $A$ be an $S$-scheme and $x$ an $A$-valued point of $X$.
      From the short exact sequence \eqref{eq:stabses} we see that the induced morphism $B\uAut_X(x) \to B\uAut_\sX(x)$ is a $\St^G_X(x)$-torsor, where $\sX = [X/G]$.
      Moreover, there is a commutative diagram
      \[\begin{tikzcd}
        B\uAut_X(x) \ar{r}\ar[twoheadrightarrow]{d}
        & X \ar[twoheadrightarrow]{d}
        \\
        B\uAut_\sX(x) \ar{r}
        & \sX
      \end{tikzcd}\]
      where the right-hand arrow is a $G$-torsor.
      In particular, we find that there is a canonical $\St^G_X(x)$-action on the group algebraic space $\uAut_X(x)$, and the canonical monomorphism
      \[ B\uAut_X(x) \hook X \]
      is equivariant with respect to the $\St^G_X(x)$-action on the source and $G$-action on the target.
    \end{rem}

  \subsection{Fixed loci}
  \label{ssec:fixed}

    Fix an fppf group algebraic space $G$ over a $k$-scheme $S$.
    Let $X$ be an Artin stack over $S$ with $G$-action.

    \begin{defn}\label{defn:oy1bpy1b1}
      We define the \emph{$G$-fixed locus} $X^G \sub X$ as follows.
      An $A$-point $x \in X(A)$, where $A$ is an $S$-scheme, belongs to $X^G$ if and only if the homomorphism of group algebraic spaces over $A$ \eqref{eq:evvTWDTvdYQvm}
      \begin{equation*}
        \alpha_A : \uAut_{[X/G]}(x) \to G_{A}
      \end{equation*}
      is surjective\footnote{%
        i.e., an effective epimorphism of fppf sheaves: it admits a section after base change along some fppf cover $A' \twoheadrightarrow A$
      }.
      In other words, $X^G$ is the locus of points $x \in X(A)$ where the inclusion of the $G$-stabilizer $\St^G_X(x) \sub G_A$ is an equality (see \defnref{defn:stabilizer}).
    \end{defn}

    \begin{defn}\label{defn:redfix}
      Suppose $X$ is $1$-Artin.
      When the subset $\abs{X^G} \sub \abs{X}$ is closed, it follows that there exists a reduced closed substack $X^G_\red$ of $X$ such that $\abs{X^G_\red} = \abs{X^G}$.
      In that case, we refer to $X^G_\red \sub X$ as the \emph{reduced $G$-fixed locus}.
    \end{defn}

    \begin{defn}\label{defn:hfix}
      The \emph{homotopy $G$-fixed point stack} of $X$, denoted $X^{hG}$, is defined as follows.
      An $A$-point of $X^{hG}$, where $A$ is an $S$-scheme, is an $A$-point $x \in X(A)$ together with a \emph{group-theoretic section} of the monomorphism \eqref{eq:evvTWDTvdYQvm}
      $$\alpha_A : \uAut_{[X/G]}(x) \to G_{A}$$
      of group algebraic spaces over $A$.
      The projection $X^{hG} \to X$ factors through the fixed locus $X^G \sub X$ by definition.
    \end{defn}

    We recall the following results from \cite[App.~A]{virloc}:

    \begin{thm}\label{thm:fixed}
      Let $S$ be a $k$-scheme and $X$ a quasi-separated $1$-Artin stack $X$ over $S$ with separated diagonal.
      Let $G \to S$ be an fppf group algebraic space with connected fibres acting on $X$ over $S$.
      Then we have:

      \begin{thmlist}
        \item\label{item:fixed/asp}
        Suppose $X$ is an algebraic space.
        Then there is a canonical isomorphism $X^G \simeq X^{hG}_\cl$ over $X$.
        Moreover, the morphisms $X^G \to X$ and $X^{hG} \to X$ are closed immersions.

        \item\label{item:fixed/redfix}
        Suppose $X$ is locally of finite type over $k$ and $1$-Artin with finite stabilizers.
        Then the subset $\abs{X^G} \sub \abs{X}$ is closed.
        In particular, the reduced $G$-fixed locus $X^G_\red \sub X$ exists.

        \item\label{item:fixed/reparam}
        Suppose $S$ is locally noetherian, $G=T \to S$ is a split torus, and $X$ is of finite type over $k$ and $1$-Artin with affine stabilizers.
        Then there exists an isogeny $\rho : T' \twoheadrightarrow T$ where $T'$ is a split torus over $S$ and a canonical surjection $X^{hT'}_\red \to X^T_\red$ over $X$, where the target is the reduced $T$-fixed locus.
        In particular, the set-theoretic image of $X^{hT'} \to X$ coincides with $\abs{X^T} \sub \abs{X}$.

        \item\label{item:fixed/reparamclosed}
        Suppose $S$ is locally noetherian, $G=T \to S$ is a split torus, and $X$ is locally of finite type over $k$ and Deligne--Mumford with quasi-compact separated diagonal.
        Then the canonical morphism $X^{hT} \to X$ is a closed immersion.
      \end{thmlist}
    \end{thm}

    \begin{proof}
      See Prop.~A.21, Prop.~A.12, Cor.~A.42, and Prop.~A.23, respectively, of \cite{virloc}.
    \end{proof}

\section{Intersection theory on stacks}\label{sec:CohIT}

  In this section we briefly summarize the definitions and main properties of motivic Borel--Moore homology of Artin stacks from \cite{KhanVirtual}.

  \subsection{Motivic Borel--Moore homology}
  \label{ssec:BM}

    \subsubsection{Motivic spectra}

      Let $\SH(k)$ denote the stable \inftyCat of motivic spectra over $\Spec(k)$; see e.g. \cite{VoevodskyICM,KhanSix} for definitions.
      We let $\Lambda \in \SH(k)$ be a motivic commutative ring spectrum which is $\Q$-linear and oriented, such as
      \begin{align*}
        \Lambda &= \Q^\mot \qquad &\text{(motivic cohomology)},\\
        \Lambda &= \KGL_\Q \qquad &\text{(algebraic K-theory)},\\
        \Lambda &= \MGL_\Q \qquad &\text{(algebraic cobordism)};
      \end{align*}
      see Examples~2.10, 2.12, and~2.13 of \cite{KhanVirtual}.
      By \cite[Cor.~14.2.16]{CisinskiDegliseBook}, we may (and will) regard $\Lambda$ as a $\Q^\mot$-algebra.

      We consider the twists
      \[ \Lambda\vb{n} := \Lambda(n)[2n] \in \SH(k) \]
      where $\Lambda(n)$ is the Tate twist by $n\in\Z$ (see \cite[Def.~2.4.17]{CisinskiDegliseBook}).

    \subsubsection{Generalized Borel--Moore chains}

      Given an Artin stack $X \in \Stk_k$, we have the complexes of \emph{Borel--Moore chains} with coefficients in $\Lambda$ (see \cite[\S 2.1]{KhanVirtual}):
      \[
        \CBM(X; \Lambda) := \RGamma(X, a_X^!\Lambda),
      \]
      where $a_X : X \to \Spec(k)$ is the structural morphism and $a_X^!$ is the $!$-inverse image functor constructed in \cite[Thm.~A.5]{KhanVirtual}\footnote{%
        or more generally in \cite{KhanWeavelisse}
      }
      We regard this as a \emph{complex}, which here is an abbreviation of ``object of the derived \inftyCat $\D(\Q)$ of chain complexes of $\Q$-vector spaces''.

      We will often omit the coefficient when it is understood, and write
      \[ \CBM(X)\vb{n} := \CBM(X; \Lambda\vb{n}) \]
      for the twists.

      Explicitly, $\CBM(X)$ can be described as a homotopy limit
      \[
        \CBM(X) \simeq \lim_{(U,u)} \CBM(U)\vb{-d_u}
      \]
      over pairs $(U,u)$ where $U\in\Stk_k$ is a scheme and $u : U \to X$ is a smooth morphism of relative dimension $d_u$.

      When $X$ is a global quotient stack $[U/G]$, $\CBM(X)$ can be described as Borel-type $G$-equivariant Borel--Moore chains on $U$; see \cite[Cor.~6.2]{Equilisse}.

    \subsubsection{Operations}

      Each $\CBM(X; \Lambda)\vb{n}$ is a module over the complex of cochains,
      \[ \Ccoh(X; \Lambda) := \RGamma(X, a_X^*\Lambda) \]
      by cap product (see \cite[2.2.6]{KhanVirtual}).

      Given a proper morphism $f : X \to Y$ in $\Stk_k$, there is a push-forward map
      \[
        f_* : \CBM(X) \to \CBM(Y)
      \]
      induced by the counit $f_*f^! \simeq f_!f^! \to \id$ (see \cite[2.2.1]{KhanVirtual}).

      Given a smooth morphism $f : X \to Y$ of relative dimension $d$ in $\Stk_k$, there is a Gysin pull-back map
      \[
        f^! : \CBM(Y) \to \CBM(X)\vb{-d}
      \]
      induced by the cotrace $\id \to f_*f^!\vb{-d}$ (see \cite[2.2.2]{KhanVirtual}).

    \subsubsection{Relative and equivariant versions}

      For a morphism $f : X \to S$ in $\Stk_k$, we consider \emph{relative Borel--Moore chains}: 
      \[
        \CBM(X_{/S}; \Lambda) := \RGamma(X, f^!(\Lambda_S)),
      \]
      where $\Lambda_S = a_S^*(\Lambda)$ for $a_S : S \to \Spec(k)$ the projection.
      As in the absolute case, these admit proper push-forward and smooth Gysin pull-back functoriality in $X$.
      Note that $\CBM(X) = \CBM(X_{/k})$ and $\Ccoh(X) = \CBM(X_{/X})$.

      Let $G$ be an fppf group algebraic space over $k$.
      For every $X \in \Stk_k$ with $G$-action, we write
      \begin{align*}
        \Ccoh_{G}(X) &:= \Ccoh([X/G]),\\
        \Chom^{\BM,G}(X) &:= \CBM([X/G]_{/BG}),
      \end{align*}
      for $G$-equivariant cochains and $G$-equivariant Borel--Moore chains on $X$.
      Since $BG$ is smooth of dimension $-\dim(G)$, we have
      \begin{equation*}
        \Chom^{\BM,G}(X) \simeq \CBM([X/G])\vb{-\dim(G)}.
      \end{equation*}

    \subsubsection{Periodization}

      We consider the commutative $\Lambda$-algebra
      \[ \Lambda\per := \Lambda[\gamma,\gamma^{-1}] \simeq \Lambda \otimes_{\Q^\mot} \Q^\mot[\gamma,\gamma^{-1}] \in \SH(k) \]
      where $\gamma : \Lambda\vb{1} \to \Lambda\per$ is a free invertible generator (cf. \cite[Cor.~14.2.17(3)]{CisinskiDegliseBook}).
      At the level of $\Lambda$-modules, we have $\Lambda\per \simeq \bigoplus_{i\in\Z} \Lambda\vb{i}$.
      Note that $\Lambda \mapsto \Lambda\per$ commutes with $*$-pullback along $X \to \Spec(k)$.

      We also consider the $\Ccoh(X)$-algebras
      \begin{align*}
        \Ccoh(X)\per &:= \Ccoh(X)[\gamma, \gamma^{-1}] := \bigoplus_{i\in\Z} \Ccoh(X)\vb{i},\\
        \quad \Ccohhat(X)\per &:= \Ccoh(X) \llbracket \gamma,\gamma^{-1}\rrbracket := \prod_{i\in\Z}\Ccoh(X)\vb{i}.
      \end{align*}
      There are canonical maps
      \begin{equation}\label{eq:forest}
        \Ccoh(X)\per
        \to \Ccoh(X; \Lambda\per)
        \to \Ccohhat(X)\per.
      \end{equation}

      We consider the periodized Borel--Moore chain complexes
      \begin{align*}
        \CBM(X)\per &:= \CBM(X)[\gamma,\gamma^{-1}],\\
        \Chomhat^\BM(X)\per &:= \CBM(X)\llbracket\gamma,\gamma^{-1}\rrbracket
      \end{align*}
      as modules over $\Ccoh(X)\per$ and $\Ccohhat(X)\per$, respectively.

      For $G$ an fppf group algebraic space over $k$, we define $\Ccoh_G(X)\per$, $\Ccohhat_G(X)\per$, $\Chom^{\BM,G}(X)\per$, and $\Chomhat^{\BM,G}(X)\per$ similarly.

    \subsubsection{Chern character}
    \label{sssec:Chern}
      
      The Chern character induces a canonical isomorphism
      \[
        \KGL_\Q \simeq \Q^\mot\per
      \]
      of $\Q^\mot$-algebras in $\SH(k)$ by \cite[Cor.~14.2.17(3)]{CisinskiDegliseBook}.
      Via the $\Q^\mot\per$-algebra structure on $\Lambda\per$ we obtain a canonical map
      \begin{equation*}
        \ch : \KGL_\Q \to \Lambda\per
      \end{equation*}
      of $\Q^\mot$-algebras.

      For every $X \in \Stk_k$ this yields a canonical map
      \begin{equation*}
        \Ccoh(X; \KGL_\Q)
        \to \Ccoh(X; \Lambda\per)
        \to \Ccohhat(X; \Lambda)\per
      \end{equation*}
      where the second arrow is as in \eqref{eq:forest}.
      Recall from \cite[Ex.~2.13]{KhanVirtual} that $\Ccoh(X; \KGL_\Q)$ is identified with the étale K-theory of $X$ (see \cite[\S 5.2]{KhanKstack}).
      Via the canonical map from K-theory to its étale localization, we thus get a Chern character map
      \begin{equation*}
        \ch : \K(X)
        \to \K^\et(X)_\Q
        \simeq \Ccoh(X; \KGL_\Q)
        \to \Ccohhat(X; \Lambda)\per.
      \end{equation*}
      In particular, K-theory classes $\alpha \in \K(X)$ act on $\Chomhat^{\BM}(X; \Lambda)\per$.

      For example, for any fppf group algebraic space $G$ over $k$, there is a ring homomorphism
      \[
        \Rep(G) \simeq \K_0(BG) \to \pi_0 \Ccohhat(BG; \Lambda)\per
      \]
      where $\Rep(G)$ is the representation ring of $G$.
      When $G$ acts on $X \in \Stk_k$, $\pi_0 \Chomhat^{\BM,G}(X; \Lambda)$ is thus a $\Rep(BG)$-module.

  \subsection{Chow, G-theory, bordism}

    \subsubsection{Chow homology}

      For $\Lambda = \Q^\mot$ the motivic cohomology spectrum, we write
      \[
        \CH_n(X, s) := \pi_s \CBM(X; \Q^\mot\vb{-n})
        \simeq \H^{-2n-s}(X; a_X^!\Q^\mot(-n))
      \]
      for every $X \in \Stk_k$ and $n,s\in \Z$, and $\CH_n(X) := \CH_n(X, 0)$.

      Suppose $X$ is a global quotient stack, e.g. $[U/G]$ where $U$ is an algebraic space of finite type over $k$ and $G$ is a linear algebraic group over $k$ acting on $U$.
      Then $\CH_n(X)$ is identified with the equivariant Chow group $\CH_{n-\dim(G)}^G(U) \otimes \Q$ of \cite{EdidinGraham} by \cite[Cor.~6.5]{Equilisse}.
      In particular, it is identified with Kresch's Chow group \cite{Kresch} in this case.

      Suppose $X$ is a $1$-Artin stack with finite inertia\footnote{%
        When $X$ is separated, this is equivalent to having finite stabilizers.
      }.
      Then $X$ admits a coarse moduli space $\pi : X \to M$ \cite{KeelMori} and the push-forward $\pi_* : \CH_n(X) \to \CH_n(M)$ is an isomorphism by proper codescent (see \cite[Thm.~5.7]{KhanKstack}).
      In particular, $\CH_*(X)$ is identified with Kresch's Chow groups also in this case.

      In fact, for any $1$-Artin stack $X$ there is a canonical map to $\CH_n(X)$ from Kresch's Chow group, specializing to the above isomorphisms.
      It will be shown in \cite{BaePark} that it is invertible as long as $X$ has affine stabilizers (and hence admits a stratification by global quotient stacks, see \cite[Prop.~2.6]{HallRydhGroups}).

    \subsubsection{G-theory}

      For $\Lambda = \KGL_\Q$ the algebraic K-theory spectrum, we write
      \[
        \G^\et_s(X) := \pi_s \CBM(X; \KGL_\Q)
      \]
      for every $X \in \Stk_k$ and $s\in \Z$.
      These are the étale G-theory groups of $X$ (see \cite[Ex.~2.13]{KhanVirtual} and \cite[\S 5.2]{KhanKstack}).
      Recall that we have $\KGL_\Q\vb{n} \simeq \KGL_\Q$ for all $n\in\Z$.
      We have the étale-local Grothendieck--Riemann--Roch isomorphism
      \[ \G^\et_s(X) \simeq \prod_{n\in\Z} \CH_n(X, s) \]
      for all $s$ by \cite[\S 3.5]{KhanVirtual}.

      For $X=[U/G]$ a quotient stack, $\G^\et_s(X)$ is Borel-equivariant G-theory $\G^{\triangleleft,G}_s(U)_\Q$ by \cite[Cor.~7.8]{Equilisse}.
      In particular, $\G^\et_0(X)$ is the \emph{completion} $\G_0^G(U)^\wedge_{I_G}$ of the Grothendieck group of $G$-equivariant coherent sheaves at the augmentation ideal $I_G$.

    \subsubsection{Bordism}

      For $\Lambda = \MGL_\Q$ the algebraic cobordism spectrum, we write
      \[
        \Omega_n(X, s) := \pi_s \CBM(X; \MGL_\Q\vb{-n})
      \]
      for every $X \in \Stk_k$ and $n,s\in \Z$, and $\Omega_n(X) := \Omega_n(X, 0)$.

      When $X = [U/G]$ is a quotient stack and $k$ is of characteristic zero, $\Omega_n(X)$ is identified with the equivariant bordism group $\Omega_{n-\dim(G)}^G(U) \otimes \Q$ of \cite{HellerMalagonLopez} by \cite[Thm.~8.4]{Equilisse}.

    \subsubsection{Betti/\texorpdfstring{$\ell$}{l}-adic Borel--Moore homology}

      The definition of Borel--Moore chains makes sense in any appropriate six functor formalism.
      Specifically, for the purposes of this paper, one can replace $\SH(-)$ by any étale-local oriented topological weave \cite{Weaves}; see also \cite{Concat}.

      For example, over $k = \bC$ we may take $\Lambda = \bQ$ in the derived \inftyCat of abelian groups, and form $\CBM(X; \bQ)$ using the $!$-pullback functor on derived \inftyCats of sheaves of abelian groups.
      See e.g. \cite[\S 1.1]{dimredcoha} for the six functor formalism on Artin stacks in this context.

      Similarly, over a field $k$ in which $\ell$ is invertible, we may form $\CBM(X; \bQ_\ell)$ using the six functor formalism for $\ell$-adic sheaves on Artin stacks \cite{LiuZheng}.
  
  \subsection{Localization triangle}

    \begin{thm}
      Let $S\in\Stk_k$, $X\in\Stk_S$, and $i : Z \to X$, $j : U \to X$ a pair of complementary closed and open immersions.
      Then there is a canonical exact triangle
      \[
        \CBM(Z_{/S})
        \xrightarrow{i_*} \CBM(X_{/S})
        \xrightarrow{j^!} \CBM(U_{/S})
      \]
      in $\D(S)$.
    \end{thm}
    \begin{proof}
      By \cite[Thm.~2.18]{KhanVirtual} we have for every motivic sheaf $\sF$ on $X$ the exact triangle
      \[
        i_*i^!(\sF) \to \sF \to j_*j^!(\sF).
      \]
      The claim follows by taking $\sF = f^!(\Lambda_S)$, where $f : X \to S$ is the structural morphism, and applying derived global sections.
    \end{proof}
    
  \subsection{Projective bundle formula}
  \label{ssec:PBF}

    \begin{thm}\label{thm:PBF}
      Let $S\in\Stk_k$ and $X\in\Stk_S$.
      Let $\cE$ be a locally free sheaf on $X$ of rank $r+1$ ($r\ge 0$), and write $\pi : P=\P(\cE) \to X$ for the associated projective bundle.
      Then the maps
      \[
        \pi^!(-) \cap e(\cO(-1))^{\cup i} : \CBM(X_{/S}) \to \CBM(P_{/S})\vb{-r+i},
      \]
      induce a canonical isomorphism of $\Ccoh(X)$-modules
      \begin{equation}\label{eq:0pgub1}
        \bigoplus_{0\le i\le r} \CBM(X_{/S})\vb{r-i} \to \CBM(P_{/S}).
      \end{equation}
    \end{thm}
    \begin{proof}
      Since $\CBM(-_{/S})$ satisfies étale descent, the map \eqref{eq:0pgub1} is the limit over smooth morphisms $t : T \to X$ with $T$ a scheme of the analogous maps
      \[
        \bigoplus_{0\le i\le r} \CBM(T_{/S})\vb{r-i-d_t} \to \CBM(T\fibprod_S P_{/S})\vb{-d_t},
      \]
      where $d_t$ is the relative dimension of $t : T \to X$.
      We may therefore assume that $X$ is a scheme.
      By Zariski descent on $X$, we may further assume that $\cE$ admits a trivial direct summand.
      It is therefore enough to show the following:

      \begin{enumerate}
        \item [($\ast$)] If \eqref{eq:0pgub1} is invertible for $\cE = \cE_0$ (locally free of rank $r$), then it is also invertible for $\cE = \cE_0\oplus\cO$.
      \end{enumerate}

      Let $P_0 = \P(\cE_0)$, $P = \P(\cE_0\oplus\cO)$, and $E=\V(\cE) = P \setminus P_0$.
      Consider the following diagram:
      \[
        \begin{tikzcd}
          \bigoplus_{0}^{r-1} \CBM(X_{/S})\vb{r-1-i} \ar{d}\ar{r}
          & \bigoplus_{0}^r \CBM(X_{/S})\vb{r-i} \ar{d}\ar{r}
          & \CBM(X_{/S})\vb{r} \ar{d} \\
          \CBM({P_0}_{/S}) \ar{r}
          & \CBM(P_{/S}) \ar{r}
          & \CBM(E_{/S})
        \end{tikzcd}
      \]
      where the upper row is split exact and the lower horizontal row is the localization triangle.
      The right-hand vertical arrow is Gysin pull-back along the projection $E \to X$ (which is invertible by homotopy invariance), and the left-hand and middle vertical arrows are \eqref{eq:0pgub1} for $\cE = \cE_0$ and $\cE = \cE_0\oplus\cO$, respectively.
      Hence it will suffice to show that this diagram is commutative.

      The inclusion $i : P_0 \hook P$ exhibits $P_0$ as the zero locus on $P$ of the cosection $\pi^*(\cO_P(-1)) \oplus \cO_P \to \cO_P$ (the projection onto the second component), so we have a canonical homotopy
      \[ i_* i^! (-) \simeq (-) \cap e(\cO_P(-1)) \]
      of maps $\CBM(P_{/S}) \to \CBM(P_{/S})\vb{1}$.
      The induced homotopy
      \[
        i_* \pi_0^! (-) \simeq \pi^!(-) \cap e(\cO_P(-1)),
      \]
      where $\pi_0 : P_0 \to X$ and $\pi : P \to X$ are the projections, then gives rise to a homotopy up to which the left-hand square in the above diagram commutes.
      Since $\cO_P(-1)$ on $P$ restricts to the trivial line bundle on $E$, there are canonical homotopies $e(\cO_P(-1))^{\cup i}|_E \simeq 0$ for all $i>0$, whence a homotopy up to which the right-hand square also commutes.
      The claim follows.
    \end{proof}

  \subsection{Reduction to the subtorus}
  \label{ssec:subtorus}

    \begin{thm}\label{thm:Weyl}
      Let $S\in\Stk_k$ and $X\in\Stk_S$.
      Suppose $G=\GL_n$ acts on $X$ and denote by $T \sub G$ the subgroup of diagonal matrices.
      Then the Gysin pull-back along the smooth morphism $p : [X/T] \to [X/G]$
      \[
        p^! : \CBM([X/G]_{/S}) \to \CBM([X/T]_{/S})\vb{-d},
      \]
      where $d=\dim(G/T)$, admits a $\Ccoh([X/G])$-module retract.\footnote{%
        For this first statement, we do not need the assumption that $\Lambda$ is $\Q$-linear.
      }
      Moreover, it induces an isomorphism
      \[
        p^! : \CBM([X/G]_{/S})\per \to \CBM([X/T]_{/S})\per^{hW},
      \]
      where $W=\Sigma_n$ is the Weyl group and $(-)^{hW}$ denotes homotopy invariants with respect to the $W$-action induced by the canonical $W$-action on $[X/T]$.
    \end{thm}
    \begin{proof}
      Since $[X/T] \to [X/B]$ is an affine bundle, where $B \sub G$ is the Borel subgroup, the Gysin map
      \[ \CBM([X/B]_{/S}) \to \CBM([X/T]_{/S}) \]
      is invertible by homotopy invariance.
      Therefore, it is enough to show that the Gysin map
      \[ \CBM([X/G]_{/S}) \to \CBM([X/B]_{/S}) \]
      admits a retract.
      Under the Morita isomorphism $[X/B] \simeq [(X \times G/B)/G]$ (with $G$ acting on $G/B$ by conjugation), this is identified with the Gysin map
      \[ f^! : \CBM([X/G]_{/S}) \to \CBM([(X\times G/B)/G]_{/S}) \]
      where $f : [(X \times G/B)/G] \to [X/G]$ is the canonical morphism.
      Note that $f$ is proper (since $G/B$ is proper over $k$), so that there is a proper push-forward $f_*$.
      We claim that this provides the desired retract.
      Indeed, since $G/B$ is the scheme of complete flags in $\A^n_X$ over $X$, the projection $f$ factors as a sequence of iterated projective bundles.
      Thus the claim follows from the fact that for a projective bundle $\pi$, the pullback $\pi^!$ admits a retract by the projective bundle formula (\thmref{thm:PBF}), cf. \cite[Ex.~3.3.5]{Fulton}.
      
      For the second statement, recall first that since we are working with $\Q$-linear coefficients, the homotopy groups of the homotopy $W$-invariants can be computed as the ordinary $W$-invariants of the homotopy groups.
      Then, as above, we are reduced to the computation of Borel--Moore chains on flag schemes.
      For this one may follow Grothendieck's computation in the Chow groups (see \cite[\S 3, Thm.~1]{GrothendieckIntersection}) word for word, again using the projective bundle formula (\thmref{thm:PBF}) as input.
    \end{proof}

\section{Concentration}
\label{sec:conc}

  In this section we assume the base ring $k$ is noetherian.
  Let $\Lambda \in \D(k)$ be an object, fixed throughout the section.
  As per our conventions recalled in \ssecref{ssec:BM}, we will leave $\Lambda$ implicit in the notation.

  \ssec{The master theorem}

    Given any set $\Sigma$ of line bundles on an Artin stack $\sX$, our main result is a stabilizer-wise criterion for torsionness of Borel--Moore chains on $\sX$ with respect to the first Chern classes of the line bundles in $\Sigma$.

    \begin{notat}\label{notat:loc}
      Let $\sX\in \Stk_k$ be quasi-compact and let $\Sigma \sub \Pic(\sX)$ be a subset.
      We consider the localizations
      \begin{equation*}
        \Ccoh(\sX)_{\Sigma\dash\loc} := \Ccoh(\sX)\per[c_1(\Sigma)^{-1}]
      \end{equation*}
      and
      \begin{equation*}
        \CBM(\sX)_{\Sigma\dash\loc} := \CBM(\sX)\per[c_1(\Sigma)^{-1}]
        \simeq \CBM(\sX) \otimes_{\Ccoh(\sX)}\Ccoh(\sX)_{\Sigma\dash\loc}.
      \end{equation*}
      We omit $\Sigma$ from the subscript when there is no risk of confusion.
    \end{notat}

    \begin{rem}
      On homotopy groups, we have
      \[ \pi_s \CBM(\sX)_{\Sigma\dash\loc} \simeq \pi_s(\CBM(\sX)\per)[c_1(\Sigma)^{-1}] \]
      by \cite[Prop.~7.2.3.25(2)]{LurieHA}.
      In particular, for $\Lambda = \Q^\mot$ and $s=0$ we recover the localization $\CH_*(\sX)_\loc$ appearing in the statement of \thmref{thm:intro/stacky conc}.
    \end{rem}

    \begin{defn}\label{defn:cond L}
      Let $\sX\in \Stk_k$ and $\Sigma \sub \Pic(\sX)$ a subset.
      We introduce the following condition on $\Sigma$:
      \begin{defnlist}[label={(L)},ref={(L)}]
        \item\label{cond:L}
        For every geometric point $x$ of $\sX$, there exists an invertible sheaf $\cL(x) \in \Sigma$ whose restriction $\cL(x)|_{B\uAut_\sX(x)}$ is trivial.
      \end{defnlist}
    \end{defn}

    \begin{rem}
      If $k$ is a field, we may also consider variants of condition~\ref{cond:L} where we look at $k$-rational points or $\bar{k}$-rational points for an algebraic closure $\bar{k}$ of $k$.
      The distinction between these versions will play no role in our proofs, so we may sometimes abuse language by referring to any of them as ``condition~\ref{cond:L}'' when there is no risk of confusion.
      The same applies for other variants of this condition that we will consider below, such as conditions~\ref{cond:K} and ~\ref{cond:L_G}.
    \end{rem}

    The main result of this section is as follows:

    \begin{thm}\label{thm:conc master}
      Let $S \in \Stk_k$, $\sX\in \Stk_S$ quasi-compact $1$-Artin with affine stabilizers, and $\Sigma \sub \Pic(\sX)$ a subset.
      If $\Sigma$ satisfies condition \itemref{cond:L} for $\sX$, then we have
      \[
        \CBM(\sX_{/S})_{\Sigma\dash\loc} = 0.
      \]
    \end{thm}

    \begin{cor}
      Let $\sX \in \Stk_k$ be quasi-compact $1$-Artin with affine stabilizers.
      If $\Sigma$ satisfies condition \itemref{cond:L} for $\sX$, then we have $\Ccoh(\sX)_{\Sigma\dash\loc} = 0$.
    \end{cor}
    \begin{proof}
      Take $S=\sX$ in \thmref{thm:conc master}.
    \end{proof}

    \begin{cor}\label{cor:conc}
      Let $S\in \Stk_k$, $\sZ,\sX \in \Stk_k$ quasi-compact $1$-Artin with affine stabilizers, and $i : \sZ \to \sX$ a closed immersion over $S$.
      Let $\Sigma \sub \Pic(\sX)$ be a subset satisfying condition \itemref{cond:L} for $\sX\setminus \sZ$.
      Then the direct image map
      \[
        i_* : \CBM(\sZ_{/S})_{\Sigma\dash\loc} \to \CBM(\sX_{/S})_{\Sigma\dash\loc}
      \]
      is invertible.
    \end{cor}
    \begin{proof}
      Follows by the localization triangle.
    \end{proof}

    We will also consider the following variant where $\Sigma$ is a set of K-theory classes.

    \begin{varnt}\label{varnt:cond K}
      Let $\Sigma \sub \K_0(\sX)$ be a set of \emph{K-theory} classes on $\sX$.
      We may consider the following variant of condition~\itemref{cond:L}:
      \begin{defnlist}[label={(K)},ref={(K)}]
        \item\label{cond:K}
        For every geometric point $x$ of $\sX$, we have
        \[\K_0(B\uAut_\sX(x))[\Sigma^{-1}]=0,\]
        where the action of $\Sigma$ is given by inverse image along $B\uAut_\sX(x) \to \sX$.
      \end{defnlist}
    \end{varnt}

    \begin{rem}
      Condition~\ref{cond:K} can be reformulated as follows: for every geometric point $x$ of $\sX$, there exists a K-theory class $\alpha(x) \in \K_0(\sX)$ which belongs to the multiplicative closure of $\Sigma$ such that $\alpha(x)|_{B\uAut_\sX(x)} = 0$.
    \end{rem}

    Recall from \sssecref{sssec:Chern} that there is a canonical action of K-theory on Borel--Moore chains via the Chern character.
    We consider the localizations at $\ch(\Sigma) \sub \pi_0 \Ccohhat(\sX)\per$:
    \begin{align*}
      \Ccohhat(\sX)_{\Sigma\dash\loc} &:= \Ccohhat(\sX)\per[\ch(\Sigma)^{-1}]\\
      \Chomhat^{\BM}(\sX)_{\Sigma\dash\loc} &:= \Chomhat^\BM(\sX)\per[\ch(\Sigma)^{-1}].
    \end{align*}
    We have the following analogue of \thmref{thm:conc master}:

    \begin{thm}\label{thm:conc master rational}
      Let $S\in \Stk_k$, $\sX\in \Stk_S$ quasi-compact $1$-Artin with affine stabilizers, and $\Sigma \sub \K_0(\sX)$ a subset.
      If $\Sigma$ satisfies condition \itemref{cond:K} for $\sX$, then we have
      \[
        \Chomhat^{\BM}(\sX_{/S})_{\Sigma\dash\loc} = 0.
      \]
    \end{thm}

  \ssec{Generic slices}

    The proof of \thmref{thm:conc master} will require a useful result of Thomason (see \cite[Thm.~4.10, Rem.~4.11]{ThomasonComp}) which we reformulate here in stacky terms for the reader's convenience.

    \begin{prop}\label{prop:slice}
      Let $T$ be a diagonalizable torus of finite type over a noetherian scheme $S$.
      Let $X$ be a reduced algebraic space of finite type over $S$ with $T$-action.
      Then there exists a $T$-invariant nonempty affine open subspace $U \sub X$ such that there is a canonical isomorphism
      \[
        [U/T] \simeq BT' \fibprod_S V
      \]
      of stacks over $BT$, for some subgroup $T' \sub T$ and affine scheme $V$ over $S$.
    \end{prop}
    \begin{proof}
      Recall that the schematic locus $X_0 \sub X$ is a nonempty open (see \cite[Tags~03JH, 03JG]{Stacks}).
      We claim that $X_0$ is $T$-invariant, i.e., the action morphism $T \fibprod_S X_0 \to X$ factors through $X_0$.
      In other words, the open immersion $T \fibprod_S X_0 \fibprod_X X_0 \to T \fibprod_S X_0$ (base change of the inclusion $X_0 \to X$) is invertible.
      We may assume that $S$ is the spectrum of an algebraically closed field $k$, and it will suffice to check this on $k$-points.
      Let $t : \Spec(k) \to T$ be a $k$-point.
      Then $t$ defines an isomorphism $t : X \to X$.
      Then $t(X_0) \sub X$ is a scheme and hence $t(X_0) \sub X_0$ (since $X_0$ is the largest open which is a scheme).
      Since $t$ was arbitrary, this shows that $X_0$ is $T$-invariant.
 
      Replacing $X$ by $X_0$, we may assume that $X$ is a scheme.
      Now the claim follows by combining \cite[Lem.~4.3]{ThomasonComp}\footnote{%
        One could skip the reductions above by simply observing that the argument of \emph{loc. cit.} goes through without the locally separatedness assumptions.
      } and \cite[Thm.~4.10, Rem.~4.11]{ThomasonComp}.
    \end{proof}

  \ssec{Proof of \thmref{thm:conc master}}
  \label{ssec:musical}

    Let the notation be as in \thmref{thm:conc master}.
    By nil-invariance of Borel--Moore homology, we may replace all stacks by their reductions.
    Since $\sX$ has affine stabilizers, it admits a stratification by global quotient stacks (see \cite[Prop.~2.6]{HallRydhGroups}), so again using the localization triangle we may assume that $\sX=[X/G]$, where $X$ is a reduced quasi-affine scheme of finite type over $k$ and $G=\GL_n$, $n\ge 0$.
    Using \thmref{thm:Weyl}, we may replace $G$ by its maximal subtorus $\bG_m^{\times n}$ and thereby assume $G$ is a split torus.
    (Note that the condition on $\Sigma$ is clearly preserved under this change of notation.)
    
    By \propref{prop:slice}, there is a nonempty $G$-invariant affine open $U$ of $X$ and a diagonalizable subgroup $H$ of $G$ such that $[U/G] \simeq Y \times BH$ with $Y$ an affine scheme.
    By noetherian induction and the localization triangle, we may replace $X$ by $U$, and then further replace $X$ by $Y$ and $G$ by $H$ so that $\sX = X\times BG$.
    (Again, note that the condition on $\Sigma$ is preserved.)

    We can assume $X$ is nonempty, so that by condition~\itemref{cond:L} there exists a geometric point $x$ of $X$ and an invertible sheaf $\cL \in \Pic(X)$ such that $\cL|_{BG_{k(x)}}$ is trivial.
    It is enough to show that $c = c_1(\cL)$ is nilpotent as an element of the ring $\pi_0 \Ccoh(\sX)\per$.
    
    Since $G$ is diagonalizable, we may write $\cL = \cL_m|_{\sX} \otimes m\cO_{BG}|_{\sX}$ where $m$ is a character of $G$ and $\cL_m$ is an invertible sheaf on $X$ (see \cite[Exp.~I, 4.7.3]{SGA3}).
    The equality $c|_{BG_{k(x)}} = 0$ thus implies that the invertible sheaf $\cL|_{BG_{k(x)}}$ is trivial, i.e., $m=1$ and $\cL \simeq \cL_1|_\sX$.
    Thus $c = c_1(\cL_1)|_\sX$.
    Since $X$ is a scheme, $c_1(\cL_m)$ is nilpotent (see e.g. \cite[Prop.~2.1.22(1)]{DegliseOrientation}), hence so is $c$.

    \begin{rem}
      One can always find a \emph{finite} subset $\Sigma_0 \sub \Sigma$ such that the conclusion of \thmref{thm:conc master} still holds.
      Indeed, note that after the reductions we only need a \emph{single} element of $\Sigma$, and in each reduction step we only need to use finitely many more elements.
    \end{rem}

  \ssec{Proof of \thmref{thm:conc master rational}}

    After the same reductions as in the proof of \thmref{thm:conc master}, we reduce to the case where $\sX = X\times BG$ with $X$ a (nonempty) reduced quasi-affine scheme and $G$ a diagonalizable group scheme.
    By condition~\itemref{cond:K}, there exists a geometric point $x$ of $\sX$ and a class $\alpha \in \Sigma$ with $\alpha|_{BG_{k(x)}}$ trivial, where $G_{k(x)} \simeq \uAut_\sX(x)$ is the base change of $G$ to the residue field $k(x)$.
    It is enough to show that $\alpha$ is nilpotent as an element of $\K_0(\sX)$.
    If $N$ denotes the group of characters of $G$, then we have $\K_0(\sX) \simeq \K_0(X) \otimes_\bZ \bZ[N]$ (see \cite[Exp.~I, 4.7.3]{SGA3}, as in \cite[Lem~4.14]{ThomasonComp}).
    Thus we may write $\alpha = \sum_i a_i \otimes b_i$, where $a_i \in \K_0(X)$ and $b_i$ are pairwise distinct elements of $N$.
    Restricting $\alpha$ along $BG_{k(x)} \to \sX$ and identifying $\K_0(BG_{k(x)}) \simeq \bZ[N]$, we get $\sum_i \rk(a_i) \cdot b_i = 0$ in $\bZ[N]$, and hence $\rk(a_i) = 0$ for all $i$.
    Thus by \cite[Exp.~VI, Prop.~6.1]{SGA6}, $a_i$ are all nilpotent as elements of $\K_0(X)$.
    In particular, $\alpha$ is also nilpotent as claimed.

  \ssec{The equivariant master theorem}
  \label{ssec:gmast}
  
    We formulate a $G$-equivariant version of condition~\itemref{cond:L}:

    \begin{defn}\label{defn:cond L_G}
      Let $G$ be an fppf group algebraic space over $k$ acting on $X \in \Stk_k$, and $\Sigma \sub \Pic(BG)$ a subset.
      We introduce the following condition on $\Sigma$:
      \begin{defnlist}[label={($\text{L}_G$)},ref={($\text{L}_G$)}]
        \item\label{cond:L_G}
        For every geometric point of $X$ there exists a rank one $G$-representation $\cL(x) \in \Sigma$ such that the $\St^G_X(x)$-representation $\cL(x)|_{B\St^G_X(x)}$ is trivial.
      \end{defnlist}
      Here $\St^G_X(x)$ is the $G$-stabilizer at the point $x$ (see \defnref{defn:stabilizer}).
    \end{defn}

    Specializing \thmref{thm:conc master} to the case of quotient stacks yields the following equivariant version of the master theorem.
    As in \notatref{notat:loc}, $(-)_{\Sigma\dash\loc}$ denotes $(-)\per[c_1(\Sigma^{-1})]$.

    \begin{cor}\label{cor:conc master equiv}
      Let $S\in \Stk_k$ and $G$ an fppf group algebraic space over $k$ acting on a $1$-Artin $X \in \Stk_S$ with affine stabilizers.
      Let $\Sigma \sub \Pic(BG)$ be a subset which satisfies condition~\itemref{cond:L_G} for $X$.
      Then we have
      \[
        \CBM([X/G]_{/S})_{\Sigma\dash\loc} = 0
      \]
      and in particular $\Chom^{\BM,G}(X)_{\Sigma\dash\loc} = 0$.
    \end{cor}
    \begin{proof}
      Apply \thmref{thm:conc master} to the quotient stack $\sX=[X/G]$ (with $\Sigma \sub \Pic(\sX)$ the inverse image of $\Sigma$ along $\sX = [X/G] \to BG$).
      Since there is a commutative square
      \[\begin{tikzcd}
        B\uAut_{\sX}(x) \ar{r}\ar{d}
        & B\St^G_X(x) \ar{d}
        \\
        \sX \ar{r}
        & BG,
      \end{tikzcd}\]
      triviality of a line bundle over $B\St^G_X(x)$ implies triviality over $B\uAut_{\sX}(x)$.
    \end{proof}

    \begin{cor}\label{cor:equiv conc}
      Let $G$ be an fppf group algebraic space over $k$.
      Let $S\in \Stk_k$ and $i : Z \to X$ a $G$-equivariant closed immersion over $S$ where $Z,X\in \Stk_S$ are quasi-compact $1$-Artin with affine stabilizers.
      Let $\Sigma \sub \Pic(BG)$ be a subset which satisfies condition~\itemref{cond:L_G} for $X\setminus Z$.
      Then the direct image map
      \[
        i_* : \CBM({[Z/G]}_{/S})_{\Sigma\dash\loc} \to \CBM({[X/G]}_{/S})_{\Sigma\dash\loc}
      \]
      is invertible.
    \end{cor}

    We similarly have an equivariant version of conditions~\itemref{cond:K}:

    \begin{varnt}
      Let $G$ act on $X\in \Stk_k$ and $\Sigma \sub \Rep(G) \simeq \K_0(BG)$ a subset.
      \begin{defnlist}[label={($\text{K}_G$)},ref={($\text{K}_G$)}]
        \item\label{cond:K_G}
        For every geometric point $x$ of $X$ we have
        \[\K_0(B\St^G_X(x))[\Sigma^{-1}]=0.\]
      \end{defnlist}
    \end{varnt}

    For $\Lambda$ rational, we may then specialize \thmref{thm:conc master rational} to the equivariant case as above.

  \ssec{Base change}

    \begin{lem}
      Let $\sX,\sX'\in\Stk_k$ be $1$-Artin and let $f : \sX' \to \sX$ be a morphism.
      Let $\Sigma \sub \Pic(\sX)$ be a subset and denote by $\Sigma' \sub \pi_0\Ccoh(\sX')\per$ its image by $f^*$.
      If $\Sigma$ satisfies condition~\itemref{cond:L} for $\sX$, then $\Sigma'$ satisfies condition~\itemref{cond:L} for $\sX'$.
    \end{lem}
    \begin{proof}
      Let $x'$ be a geometric point of $\sX'$ and consider its image $x = f(x)$ in $\sX$.
      By assumption, there exists an invertible sheaf $\cL(x) \in \Sigma$ with $c_1(\cL(x))|_{B\uAut_\sX(x)} = 0$.
      Then its inverse image $\cL(x') := f^*\cL(x)$ belongs to $\Sigma'$.
      Since there is a morphism of group schemes $\uAut_{\sX'}(x') \to \uAut_\sX(x)$, we also have $c_1(\cL(x'))|_{B\uAut_{\sX'}(x')} = 0$.
    \end{proof}

    \begin{cor}\label{cor:uab01yby}
      Let $G$ be an fppf group algebraic space over $k$ acting on $X,X'\in \Stk_k$ which are $1$-Artin and let $f : X' \to X$ be a $G$-equivariant morphism.
      If $\Sigma \sub \Pic(BG)$ satisfies condition~\itemref{cond:L_G} for $X$, then it also satisfies condition~\itemref{cond:L_G} for $X'$.
    \end{cor}

    Combining this with equivariant concentration (\thmref{cor:equiv conc}) yields:

    \begin{cor}\label{cor:leisureliness}
      Let $G$ be an fppf group algebraic space over $k$.
      Let $S\in \Stk_k$, $Z,X\in \Stk_S$ quasi-compact $1$-Artin with affine stabilizers, and $i : Z \to X$ a $G$-equivariant closed immersion over $S$.
      Let $\Sigma \sub \Pic(BG)$ be a subset which satisfies condition~\itemref{cond:L_G} for $X\setminus Z$.
      Then for every morphism $f : X' \to X$, direct image along the base change $i' : Z' = Z \fibprod_X X' \to X'$ induces an isomorphism
      \[
        i'_* : \Chom^{\BM,G}(Z')_{\Sigma\dash\loc} \to \Chom^{\BM,G}(X')_{\Sigma\dash\loc}.
      \]
    \end{cor}

    \begin{constr}\label{constr:entomostracan}
      Let $G$ be an fppf group algebraic space over $k$ with connected fibres, acting on a finite type $1$-Artin $X\in \Stk_k$ with separated diagonal and affine stabilizers.
      Suppose we are given a $G$-equivariant morphism $\pi : X \to M$ to a finite type algebraic space $M\in \Stk_k$.
      Denote by $M^G$ the fixed locus of $M$ (see \thmref{thm:fixed}\itemref{item:fixed/asp}), and form the cartesian square
      \[\begin{tikzcd}
        M^G\fibprod_M X \ar{r}\ar{d}
        & X \ar{d}{\pi}
        \\
        M^G \ar{r}
        & M.
      \end{tikzcd}\]
    \end{constr}

    \begin{cor}\label{cor:pull back moduli space}
      Let the notation be as in \constrref{constr:entomostracan}.
      For any subset $\Sigma \sub \Pic(BG)$, condition~\ref{cond:L_G} for $M\setminus M^G$ implies condition~\ref{cond:L_G} for $X \setminus (M^G \fibprod_M X)$.
      In particular, the direct image map
      \[
        i_* : \Chom^{\BM,G}(M^G \fibprod_M X)_{\Sigma\dash\loc} \to \Chom^{\BM,G}(X)_{\Sigma\dash\loc}
      \]
      is invertible.
    \end{cor}

    \begin{rem}
      For example, if $G=T$ is diagonalizable then we may take $\Sigma$ to be the set of nontrivial rank one representations $\cL \in \Pic(BT)$.
      Compare \cite[Prop.~6.9]{JoshuaRRI} for G-theory of smooth Deligne--Mumford stacks with torus action over an algebraically closed field.
    \end{rem}

    \begin{rem}
      For example, if $X$ has finite inertia (e.g. it is separated and Deligne--Mumford) then it admits a coarse moduli space $M$ that is of finite type \cite{KeelMori}.
      Moreover, the $G$-action automatically descends to $M$ by universal properties, in such a way that $\pi : X \to M$ is equivariant.
      Similarly, if $X$ admits a good or adequate moduli space $\pi : X \to M$ in the sense of Alper \cite{AlperGood,AlperAdequate} and is of finite type, then $M$ is of finite type \cite[Thm.~6.3.3]{AlperAdequate}.
    \end{rem}

  \ssec{Non-quasi-compact stacks}
  \label{ssec:nonqc}

    Let $\sX$ be a $1$-Artin stack which has affine stabilizers but is only \emph{locally} of finite type over $k$, and let $\Sigma \sub \Pic(\sX)$.
    Unfortunately, concentration in the form of \thmref{thm:conc master} does not hold for such $\sX$: we may have
    \[ \CBM(\sX)_{\loc} \neq 0 \]
    even when $\Sigma$ satisfies condition~\ref{cond:L} for $\sX$.
    In fact, we have simple counterexamples to \thmref{thm:excitor} even for torus actions (and there are only finitely many $T$-stabilizer groups appearing).

    \begin{exam}
      For every integer $n\ge 0$, consider the weight $1$ scaling action of $T=\bG_m$ on $U_n = \A^n \setminus 0$ and the element
      \[
        \alpha_n = t\cap [U_n] \in \Chom^{\BM,T}(U_n)\vb{-n+1}
        \simeq \CBM(\P_k^{n-1})\vb{-n+1}
      \]
      where
      \[
        t \in \Ccoh_T(\Spec(k))\vb{1} \simeq \Ccoh(\P_k^\infty)\vb{1}
      \]
      is the first Chern class of the tautological line bundle on $BT$, and the isomorphism is \cite[Cor.~6.1]{Equilisse}.
      By the projective bundle formula, we have $t^i \cap \alpha_n=0$ if and only if $i\ge n-1$.
      Thus if $X = \coprod_n U_n$, the element
      \[
        \alpha = (\alpha_n)_n \in \Chomhat^{\BM,T}(X)\per
        \simeq \prod_{i\in\Z} \Chom^{\BM,T}(X)\vb{i}
        \simeq \prod_{i\in\Z} \prod_{n\ge 1} \Chom^{\BM,T}(U_n)\vb{i}
      \]
      does not vanish after inverting $t$.
      Similarly, let $Y = \coprod_n U_n \times W_n$ where $W_n$ is smooth $1$-Artin of finite type over $k$, of pure dimension $-n$ with trivial $T$-action.
      Then $\beta_n = t \cap [U_n \times W_n] \in \Chom^{\BM,T}(U_n\times W_n)\vb{1}$, and the element
      \[
        \beta = (\beta_n)_n \in \Chom^{\BM,T}(Y)\vb{1}
        \simeq \prod_n \Chom^{\BM,T}(U_n \times W_n)\vb{1}
      \]
      does not vanish after inverting $t$.
    \end{exam}

    In this subsection we give a definition of ``localized Borel--Moore homology''
    \[
      \CBM(\sX)_{\loc}
    \]
    which is the same as the localization of Borel--Moore homology when $\sX$ is quasi-compact, but for which $\CBM(\sX)_\loc = 0$ still holds under condition~\ref{cond:L}.
    If we write $\sX$ as a filtered union of quasi-compact opens $\{\sX_\alpha\}_\alpha$, then we will see $\CBM(\sX)_\loc \simeq \lim_\alpha \CBM(\sX_\alpha)_\loc$.

    \begin{defn}
      Let $S \in \Stk_k$, $\sX\in\Stk_S$, and $\Sigma \sub \Pic(\sX)$.
      We define
      \[
        \CBM(\sX_{/S})_\loc := \CBM(\sX_{/S})_{\Sigma\dash\loc} := \lim_{\sX'\sub\sX} \CBM(\sX'_{/S})\per[c_1(\Sigma^{-1})]
      \]
      where the homotopy limit is taken over quasi-compact opens $\sX'\sub\sX$ (with restriction maps as the transition arrows).
    \end{defn}

    Since condition~\ref{cond:L} is stable under restriction to opens, the following non-quasi-compact generalization of \thmref{thm:conc master} is immediate from the definition:

    \begin{thm}\label{thm:master nonqc}
      Let $S \in \Stk_k$, $\sX\in\Stk_S$ which is $1$-Artin with affine stabilizers, and $\Sigma \sub \Pic(\sX)$.
      If $\Sigma$ satisfies condition~\itemref{cond:L} for $\sX$, then we have
      \[
        \CBM(\sX_{/S})_{\loc} = 0.
      \]
    \end{thm}

    We next note that this definition is compatible with \notatref{notat:loc}.

    \begin{prop}\label{prop:unclipped}
      Let $S \in \Stk_k$, $\sX\in\Stk_S$, and $\Sigma \sub \Pic(\sX)$.
      If $\sX$ is a filtered union of quasi-compact opens $\{\sX_\alpha\}_\alpha$, then there is a canonical isomorphism
      \[\CBM(\sX_{/S})_\loc \simeq \lim_\alpha \CBM({\sX_\alpha}_{/S})_\loc.\]
    \end{prop}
    \begin{proof}
      Consider the following commutative square:
      \begin{equation*}
        \begin{tikzcd}
          \lim_{\sX'} \CBM(\sX')\per[\Sigma^{-1}] \ar{r}\ar{d}
          & \lim_{\alpha} \CBM(\sX_\alpha)\per[\Sigma^{-1}] \ar{d}
          \\
          \lim_{\sX'} \lim_{\alpha} \CBM(\sX_\alpha\cap\sX')\per[\Sigma^{-1}]\ar[equals]{r}
          & \lim_{\alpha} \lim_{\sX'} \CBM(\sX_\alpha\cap\sX')\per[\Sigma^{-1}]
        \end{tikzcd}
      \end{equation*}
      where we omit the ``$_{/S}$'' from the notation for simplicity.
      The upper horizontal arrow comes from the fact that each $\sX_\alpha$ is quasi-compact.
      The vertical arrows are induced by restriction.
      The right-hand one is invertible because $\sX_\alpha \hook \sX$ is cofinal in $\{\sX \cap \sX_\alpha \hook \sX\}_\alpha$.
      The left-hand one is invertible because $\sX'$ is quasi-compact: we may choose a finite subset of opens in $(\sX_\alpha\cap\sX')_\alpha$ which cover $\sX'$, and then by filteredness there exists some large enough index $\beta$ such that $\sX_\beta\cap\sX' = \sX'$.
    \end{proof}

    \begin{cor}
      Let $S \in \Stk_k$, $\sX\in\Stk_S$, and $\Sigma \sub \Pic(\sX)$.
      If $\sX$ is quasi-compact then there is a canonical isomorphism
      \[\CBM(\sX_{/S})_\loc \simeq \CBM(\sX_{/S})\per[c_1(\Sigma)^{-1}].\]
    \end{cor}

    Note also that $\CBM(-)_\loc$ still satisfies the localization triangle:

    \begin{prop}
      Let $S \in \Stk_k$, $\sX\in\Stk_S$, and $\Sigma \sub \Pic(\sX)$.
      Then for every closed immersion $i : \sZ \hook \sX$ with open complement $j: \sU \hook \sX$, we have a canonical exact triangle
      \[
        \CBM(\sZ_{/S})_\loc
        \xrightarrow{i_*} \CBM(\sX_{/S})_\loc
        \xrightarrow{j^!} \CBM(\sU_{/S})_\loc.
      \]
    \end{prop}
    \begin{proof}
      Write $\sX$ as a filtered union of quasi-compact opens $\{\sX_\alpha\}_\alpha$ (e.g. take the partially ordered set of all quasi-compact opens).
      For each $\alpha$ we have the localization triangle
      \[
        \CBM({\sZ\cap \sX_\alpha}_{/S})\per[\Sigma^{-1}]
        \to \CBM({\sX_\alpha}_{/S})\per[\Sigma^{-1}]
        \to \CBM({\sU \cap \sX_\alpha}_{/S})\per[\Sigma^{-1}]
      \]
      since localization preserves exact triangles (as an exact functor).
      Passing to the homotopy limit (which is also exact) thus yields the claim by \corref{prop:unclipped}.
    \end{proof}

    As before we may now specialize to the equivariant case.
    Let $S\in\Stk_k$, $G$ be an fppf group algebraic space over $k$ which acts on a $1$-Artin $X\in\Stk_S$ with affine stabilizers, and $\Sigma \sub \Pic(BG)$ a subset.
    We have
    \[
      \CBM([X/G]_{/S})_{\loc}
      = \lim_{X'\sub X} \CBM([X'/G]_{/S})\per[c_1(\Sigma)^{-1}]
    \]
    where the limit is taken over quasi-compact $G$-invariant opens $X' \sub X$.
    We get the following non-quasi-compact version of \corref{cor:conc master equiv}:

    \begin{cor}\label{cor:master equiv nonqc}
      If $\Sigma \sub \Pic(BG)$ satisfies condition~\itemref{cond:L_G} for $X$, then we have
      \[
        \CBM([X/G]_{/S})_{\Sigma\dash\loc} = 0.
      \]
    \end{cor}

    For $\D$ rational, we have a parallel version of the above story for any $\Sigma \sub \K_0(\sX)$, so that we get a non-quasi-compact version of \thmref{thm:conc master rational} under condition~\ref{cond:L}.

\section{Torus concentration}
\label{sec:torus}

  In this section we fix a split torus $T$ over $k$, and we assume that $k$ is noetherian with no nontrivial idempotents.

  \ssec{The statement}
    
    \begin{thm}\label{thm:excitor}
      Let $T$ act on a $1$-Artin $X \in \Stk_k$.
      Let $Z \sub X$ be a $T$-invariant closed substack such that for every point $x$ of the complement $X\setminus Z$, $\St^T_X(x)$ is properly contained in $T_{k(x)}$.
      Then we have:
      \begin{thmlist}
        \item\label{item:excitor/toprail}
        Direct image along the closed immersion $i : Z \to X$ induces an isomorphism
        \[
          i_*: \Chom^{\BM,T}(Z)_{\Sigma\dash\loc} \to \Chom^{\BM,T}(X)_{\Sigma\dash\loc}
        \]
        of $\Ccoh(BT)_{\Sigma\dash\loc}$-modules, where $\Sigma \sub \Pic(BT)$ is the set of nontrivial line bundles on $BT$.

        \item
        If $X$ is quasi-compact then there exists a \emph{finite} subset $\Sigma_0 \sub \Sigma$ such that we may replace $\Sigma$ by $\Sigma_0$ in claim~\itemref{item:excitor/toprail}.
      \end{thmlist} 
    \end{thm}
    
    \begin{rem}
      In \thmref{thm:excitor} the condition that $\St^T_X(x) \subsetneq T_{k(x)}$ is equivalent to the condition that $\dim(\St^T_X(x)) < \dim(T_{k(x)})$, since $T_{k(x)}$ is irreducible.
    \end{rem}

    \begin{rem}\label{rem:roundeleer}
      Given a subgroup $H \subsetneq T$, consider the (diagonalizable) quotient $K = T/H$.
      Choose a nontrivial character of $K$, corresponding to a one-dimensional representation of $K$.
      The $T$-representation obtained by restriction may be regarded as a line bundle $\cL$ on $BT$ whose restriction $\cL|_{BH}$ is trivial.
    \end{rem}

    \begin{exam}\label{exam:a0pb10pb}
      Suppose that the fixed locus $X^T$ (\defnref{defn:oy1bpy1b1}) is empty, i.e., for every geometric point $x$ of $X$ the $T$-stabilizer $\St^T_X(x)$ is not equal to $T_{k(x)}$.
      Consider the set $\mfr{G}$ of subgroups $H \sub T$ for which there exists a geometric point $x$ of $X$ such that the $T$-stabilizer $\St^T_X(x)$ is equal to the base change $H_{k(x)}$.
      For every $H \in \mfr{G}$ let $\cL(H)$ be a nontrivial line bundle on $BT$ such that $\cL(H)|_{BH}$ is trivial (as in \remref{rem:roundeleer}).
      Then $\Sigma_0 = \{\cL(H)\}_{H\in\mfr{G}}$ satisfies condition~\hyperref[cond:L_G]{($\text{L}_T$)} for $X$.
    \end{exam}

    \begin{prop}\label{prop:leftish}
      Let $T$ act on a quasi-compact $1$-Artin $X \in \Stk_k$.
      Then there exists a nonempty $T$-invariant open of $X$ whose $T$-stabilizers are constant.
      In particular, the set $\mfr{G}$ (\examref{exam:a0pb10pb}) is finite.
    \end{prop}
    \begin{proof}
      Note first of all that the second statement follows from the first by noetherian induction.

      For the main statement, we may as well assume that $X$ is reduced (since the statement only involves geometric points).
      For $X$ an algebraic space, the statement follows immediately from \propref{prop:slice}.
      In general we argue as follows.
      By \cite[Tag~06QJ]{Stacks} and generic flatness of the morphism $[I_X/T] \to [X/T]$ (quotient of the projection $I_X \to X$), there exists a nonempty $T$-invariant open $X_0 \sub X$ such that $X_0$ is a gerbe (with respect to a flat finitely presented group algebraic space) over an algebraic space.
      Note also that this algebraic space is locally of finite type over $k$ (since $X$ is locally a trivial gerbe over it), hence of finite type over $k$ since it is quasi-compact (because $X$ is).
      Replacing $X$ by $X_0$ we may therefore assume that $X$ is a gerbe over an algebraic space $M$ of finite type over $k$.

      The $T$-action on $X$ descends along $\pi : X \to M$, for example because it is a coarse moduli space.
      By the algebraic space case, it will suffice to show that $\pi : X \to M$ is $T$-stabilizer-preserving, i.e., for every geometric point $x$ of $X$ the canonical morphism
      \[
        \St^T_X(x) \to \St^T_M(\pi(x))
      \]
      is invertible.
      Since this is a homomorphism of subgroups of $T_{k(x)}$, it is enough to show surjectivity.
      A geometric point of $\St^T_M(\pi(x))$ is a geometric point $t$ of $T$ such that $t \cdot \pi(x) = \pi(x)$.
      But this means precisely that there is an identification $t \cdot x \simeq x$ as geometric points of $X$, since on geometric points $\pi$ exhibits $M$ as the set of connected components of $X$.
    \end{proof}

    \begin{proof}[Proof of \thmref{thm:excitor}]
      Let $\Sigma_0 \sub \Sigma$ be as in \examref{exam:a0pb10pb}; when $X$ is quasi-compact this is finite by \propref{prop:leftish}.
      Thus it will suffice to show that $i_*$ becomes invertible after inverting $c_1(\Sigma_0)$.
      By construction, $\Sigma_0$ satisfies condition~\hyperref[cond:L_G]{($\text{L}_T$)} for $X\setminus Z$, so the claim follows directly from \corref{cor:master equiv nonqc} at least when $X$ has affine stabilizers.

      Otherwise, by the localization triangle we may replace $X$ with $X\setminus Z$ and thereby reduce to showing that $\Chom^{\BM,T}(X)_\loc=0$ when $\Sigma_0$ satisfies condition~\hyperref[cond:L_G]{($\text{L}_T$)} for $X$.
      As in the proof of \propref{prop:leftish}, using the localization triangle again, we may assume that there exists a $T$-equivariant stabilizer-preserving morphism $\pi: X \to M$ which exhibits $X$ as a gerbe over an algebraic space $M$ of finite type over $k$.
      Then $\Sigma_0$ satisfies condition \hyperref[cond:L_G]{($\text{L}_T$)} for $M$ as it does for $X$ and $\pi$ is a surjection on geometric points. 
      Since $M$ is an algebraic space, it follows from \propref{prop:slice} that there exists a nonempty $T$-invariant affine open $U$ of $M$ and an invertible sheaf $\cL \in \Sigma_0$ such that $c_1(\cL)$ is nilpotent as an element of the ring $\pi_0 \Ccoh_T(U)\per$ (see proof of \thmref{thm:conc master}).
      Therefore $V = \pi^{-1}(U)$ is a nonempty $T$-invariant open in $X$ such that $c_1(\cL)$ is nilpotent in $\pi_0 \Ccoh_T(V)\per$.
      In particular 
      \[
        \Chom^{\BM,T}(V)\per[c_1(\cL)^{-1}] = 0
      \]
      as it is a module over $\Ccoh_T(V)\per[c_1(\cL)^{-1}] = 0$.
      The claim now follows by noetherian induction using the localization triangle.
    \end{proof}

  \ssec{Example: fixed loci}
  \label{ssec:flannel}

    In \cite[App.~A]{virloc} we defined a $T$-fixed locus $X^T \sub X$ for an Artin stack $X$.
    In terms of $X^T$, the condition in \thmref{thm:excitor} is that the open complement $X\setminus Z$ is contained in $X \setminus X^T$.
    A priori, the substack $X^T \sub X$ may not be closed in general.
    Nevertheless we obtain the following variants of concentration using the reduced fixed locus (\defnref{defn:redfix}) or homotopy fixed point stack (\defnref{defn:hfix}) when $X$ has finite stabilizers or is Deligne--Mumford, respectively.

    \begin{cor}\label{Prop:Concentration_for_geom_fixed_locus}
      Let the notation be as in \thmref{thm:excitor} and suppose $X$ is of finite type with separated diagonal and finite stabilizers.
      Denote by $X^T_\red$ the reduced $T$-fixed locus (\defnref{defn:redfix}).
      Then direct image along the closed immersion $i : X^T_\red \to X$ (\thmref{thm:fixed}\itemref{item:fixed/redfix}) induces an isomorphism
      \[
        i_*: \Chom^{\BM,T}(X^T_\red)_{\Sigma\dash\loc} \to \Chom^{\BM,T}(X)_{\Sigma\dash\loc}
      \]
      of $\Ccoh(BT)_{\Sigma\dash\loc}$-modules.
    \end{cor}

    In the Deligne--Mumford case we can alternatively use the homotopy fixed point stack.

    \begin{cor}\label{cor:semireflex}
      Let the notation be as in \thmref{thm:excitor} and suppose $X \in \Stk_k$ is quasi-compact Deligne--Mumford with separated diagonal.
      Choose a reparametrization $T' \twoheadrightarrow T$ as in \thmref{thm:fixed}\itemref{item:fixed/reparam} and denote by $X^{hT'}$ the homotopy fixed point stack with respect to the induced $T'$-action (\defnref{defn:hfix}).
      Then direct image along the closed immersion $i : X^{hT'} \to X$ (\thmref{thm:fixed}\itemref{item:fixed/reparamclosed}) induces an isomorphism
      \[
        i_*: \Chom^{\BM,T}(X^{hT'})_{\Sigma\dash\loc} \to \Chom^{\BM,T}(X)_{\Sigma\dash\loc}
      \]
      of $\Ccoh(BT)_{\Sigma\dash\loc}$-modules.
    \end{cor}

  \ssec{Example: vector bundles and abelian cones}

    For future use we record the following example:

    \begin{prop}\label{prop:cone stab}
      Let $X \in \Stk_k$ be $1$-Artin with finite stabilizers, regarded with trivial $T$-action.
      Let $\cE \in \Coh^T(X) \simeq \Coh(X \times BT)$ be a $T$-equivariant coherent sheaf on $X$.
      Write $E = \V_X(\cE)$ for the associated cone over $X$ with $T$-action, and $E\setminus X$ for the complement of the zero section.
      If $\cE$ has no fixed part, i.e., $\cE^\fix \simeq 0$, then for every point $v$ of $E$, we have $\St^T_E(v) = T_{k(v)}$ if and only $v$ belongs to the zero section.
    \end{prop}
    \begin{proof}
      Suppose first that $X$ is the spectrum of a field.
      Since $T$ acts trivially on $X$, the $T$-representation $\cE$ splits as a direct sum of $1$-dimensional representations.
      Since $\cE$ has no fixed part, the latter have nonzero weights, so the claim is clear in this case.
      
      Now consider the general case, i.e., $X$ is $1$-Artin with finite stabilizers.
      Let $v$ be a field-valued point of $E$.
      It is clear that if $v$ belongs to the image of the zero section then it has $\St^T_E(v) = T_{k(v)}$.
      Conversely, suppose that $v \in E \setminus X$.
      To show that the inclusion $\St^T_E(v) \sub T_{k(v)}$ is proper it will suffice to show that $\dim(\St^T_E(v)) < \dim(T_{k(v)})$ (since the scheme $T_{k(v)}$ is irreducible).
      Note that $E$ also has finite stabilizers, since it is affine over $X$.
      By the short exact sequence of group schemes over $k(v)$ \eqref{eq:stabses}
      \[
        1 \to \uAut_E (v) \to \uAut_{\sE}(v) \to \St^T_E(v) \to 1,
      \]
      where $\sE = [E/T]$ is the quotient, it will moreover suffice to show that $\dim(\uAut_{\sE}(v)) < \dim(T_{k(v)})$.

      Let $x$ denote the projection of $v$ in $X$, $E_x$ the fibre of $E$ over $x$, and $\sE_x = [E_x/T]$ the quotient.
      We have another short exact sequence
      \[
        1 \to \uAut_{\sE_x}(v) \to \uAut_{\sE}(v) \to \uAut_X(x),
      \]
      by applying \remref{rem:Yana} twice, to the representable morphism $\sE_x \to \sE$ and to the morphism $\sE \to X \times BT \to X$.
      By the special case of our claim where the base $X$ is $\Spec(k(x))$, we have $\dim(\uAut_{\sE_x}(v)) < \dim(T_{k(v)})$.
      Since $\uAut_{X}(x)$ is finite, it follows that $\dim(\uAut_\sU(v)) < \dim(T_{k(v)})$ as claimed.
    \end{proof}

    By \thmref{thm:excitor} we get:

    \begin{cor}\label{cor:cone LG}
      Let the notation be as in \propref{prop:cone stab}.
      Let $\Sigma \sub \Pic(BT)$ be the subset of nontrivial line bundles.
      Then $\Sigma$ satisfies condition~\hyperref[cond:L_G]{$(\text{L}_{T})$} for $E \setminus X$.
    \end{cor}

    \begin{cor}\label{cor:conc bundle}
      Let $S\in\Stk_k$, $X\in \Stk_S$ with finite stabilizers.
      Let $\cE$ be a $T$-equivariant coherent sheaf on $X$ with no fixed part, and write $E = \V_X(\cE)$.
      Then direct image along the zero section $0 : X \to E$ induces an isomorphism of $\Ccoh(BT)_{\Sigma\dash\loc}$-modules
      \begin{equation*}
        0_* : \CBM({X \times BT}_{/S})_{\Sigma\dash\loc}
        \to \CBM([E/T]_{/S})_{\Sigma\dash\loc},
      \end{equation*}
      where $\Sigma$ is as in \corref{cor:cone LG}.
    \end{cor}
    \begin{proof}
      Follows from Corollaries~\ref{cor:cone LG} and \ref{cor:equiv conc}.
    \end{proof}

    We warn the reader that the finite stabilizers assumption in \corref{cor:conc bundle} is necessary, as we can already see in the ``simplest'' example of an Artin stack with infinite stabilizers, i.e. $B\Gm$.
    Indeed, even the conclusion of \propref{prop:cone stab} does not hold in this example:

    \begin{exam}
      Let $Z = B\Gm$, $X = [\A^1/\bG_m]$ where $\Gm$ acts with weight $1$, and $i : Z \to X$ the closed immersion induced by $0 : \Spec(k) \to \A^1$.
      Let $T=\bG_m$ act on $X$ with weight 1, so that $i$ is $T$-equivariant.
      Then the direct image map
      \[
        i_* : \Chom^{\BM,T}(B\Gm)_{T\dash\loc} \to \Chom^{\BM,T}([\A^1/\bG_m])_{T\dash\loc}
      \]
      is \emph{not} an isomorphism, where we write $T\dash\loc$ to emphasize that the localization is with respect to the $T$-action.
      Indeed, using the localization triangle (which passes to equivariant localizations as $(-)_\loc$ is an exact functor), its cofibre is the complex of $T$-equivariant Borel--Moore chains
      \[
        \Chom^{\BM,T}([\bG_m/\bG_m])_{T\dash\loc}
        \simeq \Chom^{\BM,T}(\pt)_{T\dash\loc}
        \simeq \Chom^{\BM}(BT)_{T\dash\loc}
      \]
      where $\pt = \Spec(k)$.
      This is of course nonzero: it is the complex $\Ccoh(BT)_{T\dash\loc} \simeq \Ccoh(\pt)\per[t,t^{-1}]$.
    \end{exam}

\ssec{Example: Higgs sheaves on curves}

    Let $k$ be a field and $C$ a smooth proper and geometrically connected curve over $k$.
    Consider the moduli stack $\uCoh(C)$ of coherent sheaves on $C$.
    This is a smooth $1$-Artin stack locally of finite type over $k$.

    Denote by $\Higgs$ the moduli stack of Higgs sheaves on $C$, i.e., the cotangent bundle of $\uCoh(C)$ (which for us means the total space of the cotangent \emph{complex}).
    This admits a canonical scaling action by the torus $T = \Gm$.
    Let $\bLambda$ denote the closed substack parametrizing \emph{nilpotent} Higgs sheaves, i.e., pairs $(\cF, \theta)$ where $\cF \in \Coh$ and $\theta : \cF \to \cF \otimes K_C$ is nilpotent (with $K_C$ the canonical bundle of $C$).

    \begin{thm}\label{thm:Higgs}
      Let $\Sigma \sub \Pic(BT)$ denote the set of nontrivial invertible sheaves on $BT$.
      Then the direct image map
      \begin{equation*}
        \Chom^{\BM,T}(\bLambda)_\loc \to \Chom^{\BM,T}(\Higgs)_\loc
      \end{equation*}
      is invertible (where the localization is as in \ssecref{ssec:nonqc}).
    \end{thm}
    \begin{proof}
      Let $\cF$ be a Higgs sheaf over an algebraically closed extension field $\kappa$ of $k$.
      By \thmref{thm:excitor} it is enough to show that, if $\cF$ is not nilpotent, then the $T$-stabilizer at the corresponding geometric point of $\Higgs$ is a proper subgroup of $T_{\kappa}$.
      We are grateful to A.~Minets for providing the following argument.

      First note that if $\cF$ is not nilpotent, then either its maximal torsion Higgs subsheaf $\cF_0$ is nilpotent or the quotient $\cF' = \cF/\cF_0$ is nilpotent.
      Thus we may assume that $\cF$ is either torsion or locally free.

      Recall that the moduli stack of locally free Higgs sheaves (= Higgs bundles) of rank $r$ admits a canonical $T$-equivariant map (the Hitchin fibration) to the scheme $A_r = \bigoplus_{i=1}^{r} \H^0(C, K_C^{\otimes i})$, where $T$ acts on $\H^0(C, K_C^{\otimes i})$ with weight $i$, by sending $\cF$ to the coefficients of the characteristic polynomial of its Higgs field.
      Moreover, a locally free Higgs sheaf $\cF$ is nilpotent if and only if its image in the Hitchin base is trivial.
      By \corref{cor:pull back moduli space}, this implies the claim in the locally free case.

      In the torsion case we argue as follows.
      Under the ``BNR correspondence'' (see \cite{BNRConstruction} or \cite[Lem.~6.8]{SimpsonNHTII}), Higgs torsion sheaves $\cF$ are push-forwards of coherent sheaves on $T^*_C$ of finite length.
      In particular, there is a ``Hilbert--Chow'' map from the moduli of torsion Higgs sheaves to $\Sym^d(T^*_C)$, sending $\cF$ to the support (counted with multiplicities) of the corresponding sheaf on $T^*_C$.
      Now $\cF$ is nilpotent if and only if the corresponding sheaf on $T^*_C$ is supported in the zero section, hence if and only if its image in $\Sym^d(T^*_C)$ is fixed under the $\bG_m$-scaling action.
      We conclude by \corref{cor:leisureliness}.
    \end{proof}

    \begin{rem}
      For the substack of torsion Higgs sheaves, \thmref{thm:Higgs} recovers \cite[Cor.~4.3]{Minets} by taking $\pi_0$.
      In other words, we can regard the result as a generalization of \emph{loc. cit.} to arbitrary rank Higgs sheaves\footnote{%
        A similar claim is made in \cite[Prop.~3.7]{SalaSchiffmann}, but the proof is not correct because the localization does not commute with cofiltered limits.
        This is related to the reason why the discussion in \ssecref{ssec:nonqc} is necessary.
      } and to ``higher'' oriented Borel--Moore homology theories.
    \end{rem}

    \begin{rem}
      \thmref{thm:Higgs} admits an analogous statement for Higgs $G$-bundles, for a connected reductive group $G$, using the same Hitchin fibration argument.
      Similarly, there is a parabolic variant using the parabolic Hitchin fibration \cite{Yokogawa}.
    \end{rem}

%%%%%%%%%%%%%%%%%%%%%%%%%%%%%%%%%%%%%%%%%%%%%%%%%%%%%%%%%%%%%%%%%%%%%%%%%%%

%!TEX root = constack.tex

\bibliographystyle{halphanum}

Fakultät Mathematik, Universität Duisburg-Essen, Essen, 45127, Germany

Institute of Mathematics, Academia Sinica, Taipei, 10617, Taiwan

Mathematical Institute, University of Oxford, Oxford, OX2 6GG, United Kingdom

Department of Mathematical Sciences, Seoul National University, Seoul, 08826, Korea

School of Mathematics, Tata Institute of Fundamental Research, Mumbai, 400005
India

\end{document}